\documentclass{article}
\usepackage{amsthm}
\usepackage{amsmath}
\usepackage{amsfonts}
\usepackage{mathrsfs}
\usepackage{array}
\usepackage{amssymb}
\usepackage{units}
\usepackage{graphicx}
\usepackage{tikz-cd}
\usepackage{nicefrac}
\usepackage{hyperref}
\usepackage{bbm}
\usepackage{color}
\usepackage{tensor}
\usepackage{tipa}
\usepackage{bussproofs}
\usepackage{ stmaryrd }
\usepackage{ textcomp }
\usepackage{leftidx}
\usepackage{afterpage}
\usepackage{varwidth}
\usepackage{tasks}
\usepackage{ cmll }
\usepackage{makecell}
\usepackage{MnSymbol}
\usepackage{adjustbox}
\usepackage{multirow}
\usepackage{booktabs}
\usepackage{xparse}
\usepackage{calc}
\usepackage{stackengine}
\usepackage{csquotes}
\usepackage{enumitem}
\usepackage{mathdots}
\usepackage{quiver}
\usepackage[all]{xy}

\theoremstyle{plain}
\newtheorem{thm}{Theorem}[subsection] 
\newtheorem{proposition}[thm]{Proposition}
\newtheorem{lemma}[thm]{Lemma}

\newtheorem{cor}[thm]{Corollary}

\theoremstyle{definition}
\newtheorem{defn}[thm]{Definition} 

\newtheorem{remark}[thm]{Remark}

\newtheorem{example}[thm]{Example}

\newcommand{\bb}[1]{\mathbb{#1}}

\newcommand{\comment}[1]{}
\newcommand{\lto}{\longrightarrow}

\newcommand{\prom}{\operatorname{Prom}}
\newcommand{\ctr}{\operatorname{Ctr}}
\newcommand{\der}{\operatorname{Der}}
\newcommand{\weak}{\operatorname{Weak}}

\newcommand{\cut}{\operatorname{Cut}}
\newcommand{\ax}{\operatorname{Ax}}
\newcommand{\conc}{\operatorname{Conc}}
\newcommand{\pax}{\operatorname{Pax}}

\usepackage[margin=2.5cm]{geometry}

{\gdef\scalefactor{#1}\begin{center}\proofSkipAmount \leavevmode}%
	{\scalebox{\scalefactor}{\DisplayProof}\proofSkipAmount \end{center} }

\title{Linear Logic and the Hilbert Scheme}
\author{William Troiani, Daniel Murfet}

\begin{document}

\maketitle

\begin{abstract}
We introduce a geometric model of shallow multiplicative exponential linear logic (MELL) using the Hilbert scheme. Building on previous work interpreting multiplicative linear logic proofs as systems of linear equations, we show that shallow MELL proofs can be modeled by locally projective schemes. The key insight is that while multiplicative linear logic proofs correspond to equations between formulas, the exponential fragment of shallow proofs corresponds to equations between these equations. We prove that the model is invariant under cut-elimination by constructing explicit isomorphisms between the schemes associated to proofs related by cut-reduction steps. A key technical tool is the interpretation of the exponential modality using the Hilbert scheme, which parameterizes closed subschemes of projective space. We demonstrate the model through detailed examples, including an analysis of Church numerals that reveals how the Hilbert scheme captures the geometric content of promoted formulas. This work establishes new connections between proof theory and algebraic geometry, suggesting broader relationships between computation and scheme theory.
\end{abstract}


\section{Introduction}

Linear logic \cite{LinearLogic} can be viewed as a language of constructions in the operators
\[
\oplus, \otimes, \multimap, !\,.
\]
Most of these have meanings that are familiar from tensor algebra: $\oplus$ has the semantics of ``direct sum'', $\otimes$ is ``tensor product'' and $\multimap$ is ``space of linear maps''. The exception is the \emph{exponential} $!$ which has the semantics of a cofree coalgebra \cite{hyland2003glueing, murfet2014logic, murfet2017sweedlers}. The upshot is that linear logic gives a formal language for constructions in a canonical nonlinear extension of tensor algebra. While these operators are familiar to mathematicians, the fact that the structural maps associated to their universal properties can be used to encode \emph{algorithms} is less familiar, and quite remarkable. This is a variant on the Curry-Howard-Lambek correspondence \cite{Howard1980-HOWTFN-2,lambek1988introduction} and a consequence of the fact that proofs in intuitionistic linear logic can be interpreted as algorithms.

In particular, this means that we can translate algorithms into any category of mathematical objects where $\oplus, \otimes, \multimap, !$ can be interpreted. While the class of algorithms that can be encoded in (first-order, intuitionistic) linear logic is limited, it does include the execution of a Turing machine for a finite number of steps \cite{girard1994light,clift2020encodings} and the interpretation of linear logic in vector spaces \cite{murfet2017sweedlers} has been used to make connections between Turing machines, the Ehrhard-Regnier derivative \cite{ehrhard2003differential} and statistical learning theory \cite{clift2019derivativesturingmachineslinear,clift2021geometryprogramsynthesis}.
\\

In this paper we continue a project initiated in \cite{AlgPnt}, which aims to find new interpretations of linear logic using algebraic geometry. At a conceptual level the motivation is the simple idea that the structure of a proof lies in the pattern of repeated occurrences of some atomic degrees of freedom, which are somewhat implicit in sequent calculus or proof net presentations, but which are made explicit as ``variables'' in presentations like the lambda calculus \cite{GMZ}. We can think of such a pattern as constructed from a set of atoms $x,y,z,w,\ldots$ by a set of equations $x = y, y = z, z = w, \ldots$ and in that case, why not model these equations by an ideal $(x - y, y - z, z - w, \ldots)$ and then by a scheme? In this perspective the role of sequents in a proof tree (or edges in a proof net) is to introduce the atomic degrees of freedom $x,y,z,w,\ldots$ and the role of deduction rules is to bind these atoms to each other by equations \cite{AlgPnt} which determine a geometric object.

This idea is already mildly interesting in the case of multiplicative linear logic proofs (those involving only $\otimes, \multimap$) but the really interesting question is: what kind of equations, and thus geometry, represent the deduction rules in linear logic associated to the exponential?
\\

In this paper we give an answer to this question for a set of proof nets $\pi$ which we call \emph{shallow}, for linear logic with the connectives $\otimes, \multimap, !$ (known as MELL), using the Hilbert scheme which is the new ingredient necessary to interpret the exponential connective. To each shallow proof $\pi$ we associate a closed immersion
\begin{equation}\label{eq:loc_proj_pair}
\bb{X}(\pi) \longrightarrow \bb{S}(\pi)
\end{equation}
 of schemes (Definition \ref{defn:xpi}). Schemes are how we think about sets of solutions of polynomial equations in modern algebraic geometry. Here $\bb{S}(\pi)$, which we refer to as the \emph{ambient scheme} of $\pi$, is a disjoint union of projective spaces $\mathbb{P}^n_{\mathbbm{k}}$ over the base field $\mathbbm{k}$ for various $n$. The ambient scheme depends only on the formulas labeling the edges of the proof net (or equivalently the formulas appearing in sequents in the proof tree, if we think in terms of sequent calculus), whereas the full structure of $\pi$ is reflected in the closed subscheme $\bb{X}(\pi)$. As a closed subscheme of a locally projective scheme, $\bb{X}(\pi)$ is itself locally projective. In the conceptual picture introduced above, $\bb{S}(\pi)$ introduces the atomic degrees of freedom (the coordinates in the projective spaces) and $\bb{X}(\pi)$ says how the structure of $\pi$ dictates that these degrees of freedom should be related to one another so that the geometry reflects the proof.

Our main theorem (Theorem \ref{thm:main}) says that the locally projective pair \eqref{eq:loc_proj_pair} is an invariant of proof nets, in the sense that we associate to any cut-reduction step $\gamma: \pi \lto \pi'$ a commutative diagram
\begin{equation}
\xymatrix@C+2pc{
\bb{S}(\pi) \ar[r]^{S_\gamma} & \bb{S}(\pi')\\
\bb{X}(\pi) \ar[u] \ar[r]_-{\cong} & \bb{X}(\pi') \ar[u]
}
\end{equation}
in the category of schemes, where the bottom row is an isomorphism. This is the sense in which we mean that our construction is an interpretation of a (fragment of) linear logic. We expect that these ideas extend to arbitrary proof nets, but since this seems to require more sophisticated algebraic geometry we feel it is worthwhile presenting the simple fragment separately (see Section \ref{sec:future}).
\\

While understanding the details about Hilbert schemes in this paper requires some nontrivial background in algebraic geometry, the geometry that is associated to the deduction rules involving exponentials is ultimately an expression of a simple idea. Promoting multiplicative proofs, which geometrically are interpreted by linear polynomials $x - y, y - z, z - w, \ldots$ introduces additional atomic degrees of freedom $\theta, \phi, \psi, \ldots$ which we can think of as parametrising a \emph{space of equations}
\begin{equation}
x - \theta y - \phi z, \, y - \psi w - \kappa t, \ldots
\end{equation}
in the sense that the point $\theta = 1, \phi = 0, \psi = 1, \kappa = 0$ in this parameter space stands for the system of equations $x = y, y = w$ while the point $\theta = 0, \phi = 1, \psi = 1, \kappa = 0$ stands for $x = z, y = w$. The deduction rules like the contraction rule which operate on exponentiated formulas introduce equations between these parameters such as $\theta = \psi$, which bind the identity of some equations to that of other equations (see Remark \ref{rmk:eq_btw_eq}). In this way our interpretation of shallow proof nets in locally projective pairs of schemes realises the exponential as having the semantics of a \emph{space of proofs} and the geometric content of the deduction rules involving exponentiated formulas as \emph{equations between equations}.

\section{Shallow Proofs and the Hilbert scheme}
\label{sec:parameter}
Proofs in Multiplicative Linear Logic (MLL) can be modelled by systems of linear equations between occurrences of formulas, and computation of a program is in turn modelled by the elimination of variables appearing in these systems \cite{AlgPnt}. This paper proves that \emph{shallow} proofs (Definition \ref{def:shallow_proof}) can be modelled by locally projective schemes. Algebraically, these locally projective schemes describe equations between formulas along with equations between these equations, as made precise in Remark \ref{rmk:eq_btw_eq}.

\begin{defn}
\label{def:shallow_formula}
Let $A$ be a formula. We define the \textbf{depth of $A$}, $\operatorname{Depth}(A)$, by induction on the structure of $A$ as follows:
\begin{itemize}
    \item If $A$ is atomic then $\operatorname{Depth}(A) = 0$.
    \item If $A = A_1 \boxtimes A_2$ where $\boxtimes \in \{\otimes, \parr\}$ then $\operatorname{Depth}(A) = \operatorname{max}\{\operatorname{Depth}(A_1), \operatorname{Depth}(A_2)\}$.
    \item If $A = \neg A'$ then $\operatorname{Depth}(A) = \operatorname{Depth}(A')$.
    \item If $A = \square A'$ where $\square \in \{!,?\}$ then $\operatorname{Depth}(A) = \operatorname{Depth}(A') + 1$.
\end{itemize}
\end{defn}

\begin{defn}
\label{def:shallow_proof}
A formula $A$ is \textbf{linear} if $\operatorname{Depth}(A) = 0$ and is \textbf{shallow} if $\operatorname{Depth}(A) \leq 1$. A proof $\pi$ is \textbf{linear} if all its formulas are.
\end{defn}

\begin{defn}
\label{def:nearly_linear}
    A proof $\pi$ is \textbf{pre-nearly linear} if the following hold:
    \begin{itemize}
        \item All conclusions to all Axiom-links are atomic.
        \item The conclusions of $\pi$ are of the following form: $?A_1, \ldots, ?A_n, B$ (where we allow for the possibility that $n = 0$), with $A_1, \ldots, A_n, B$ linear.
        \item All edges of $\pi$ are labelled by shallow formulas.
        \item There are no Weakening-links and there are no Promotion-links in $\pi$.
    \end{itemize}
    The \textbf{linear part} of a pre-nearly linear proof net $\pi$ is given by removing all Dereliction-links from $\pi$ along with everything beneath these Dereliction-links and attaching the premises of these Dereliction-links to Conclusion-links.

    A proof net $\pi$ is \textbf{nearly linear} if it is pre-nearly linear and all persistent paths of the linear part of $\pi$ go through $B$.
\end{defn}

\begin{defn}
\label{def:shallow}
    A proof is \textbf{shallow} if it satisfies the following:
    \begin{itemize}
        \item All of its edges are labelled by shallow formulas.
        \item There are no nested boxes.
        \item The interior of all boxes are nearly linear proof nets.
    \end{itemize}
\end{defn}

\begin{example}
Let $X$ be atomic. The following proof net is shallow:
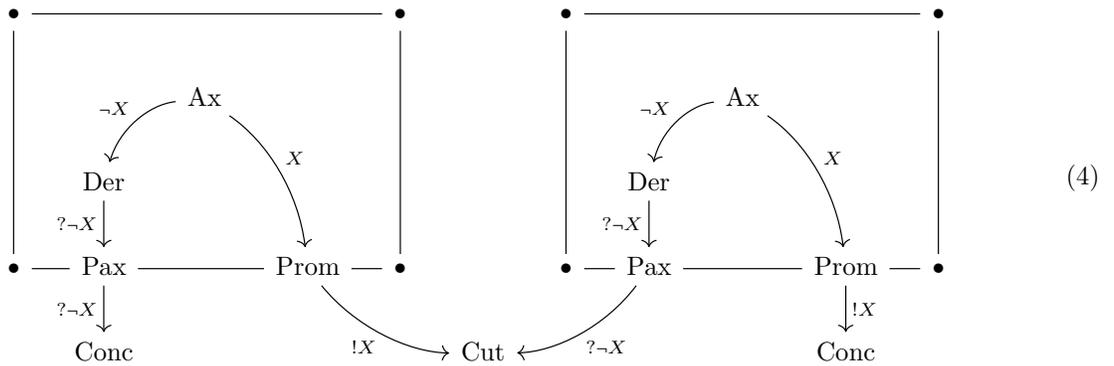
\begin{equation}
\label{eq:problematic}
\begin{tikzcd}[column sep = small]
	\bullet &&&& \bullet && \bullet &&&& \bullet \\
	&& \ax &&&&&& \ax \\
	& \der &&&&&& \der \\
	\bullet & \pax && \prom & \bullet && \bullet & \pax && \prom & \bullet \\
	& \conc &&&& \cut &&&& \conc
	\arrow[no head, from=1-1, to=4-1]
	\arrow[no head, from=1-5, to=1-1]
	\arrow[no head, from=1-5, to=4-5]
	\arrow[no head, from=1-7, to=1-11]
	\arrow[no head, from=1-7, to=4-7]
	\arrow[no head, from=1-11, to=4-11]
	\arrow["{\neg X}"', curve={height=12pt}, from=2-3, to=3-2]
	\arrow["X", curve={height=-12pt}, from=2-3, to=4-4]
	\arrow["{\neg X}"', curve={height=12pt}, from=2-9, to=3-8]
	\arrow["X", curve={height=-12pt}, from=2-9, to=4-10]
	\arrow["{?\neg X}"', from=3-2, to=4-2]
	\arrow["{?\neg X}"', from=3-8, to=4-8]
	\arrow[no head, from=4-1, to=4-2]
	\arrow[no head, from=4-2, to=4-4]
	\arrow["{?\neg X}"', from=4-2, to=5-2]
	\arrow[no head, from=4-4, to=4-5]
	\arrow["{!X}"', curve={height=12pt}, from=4-4, to=5-6]
	\arrow[no head, from=4-7, to=4-8]
	\arrow[no head, from=4-8, to=4-10]
	\arrow["{?\neg X}", curve={height=-12pt}, from=4-8, to=5-6]
	\arrow[no head, from=4-10, to=4-11]
	\arrow["{!X}", from=4-10, to=5-10]
\end{tikzcd}
\end{equation}
The following proof net is \emph{not} shallow even though all of its formulas are:
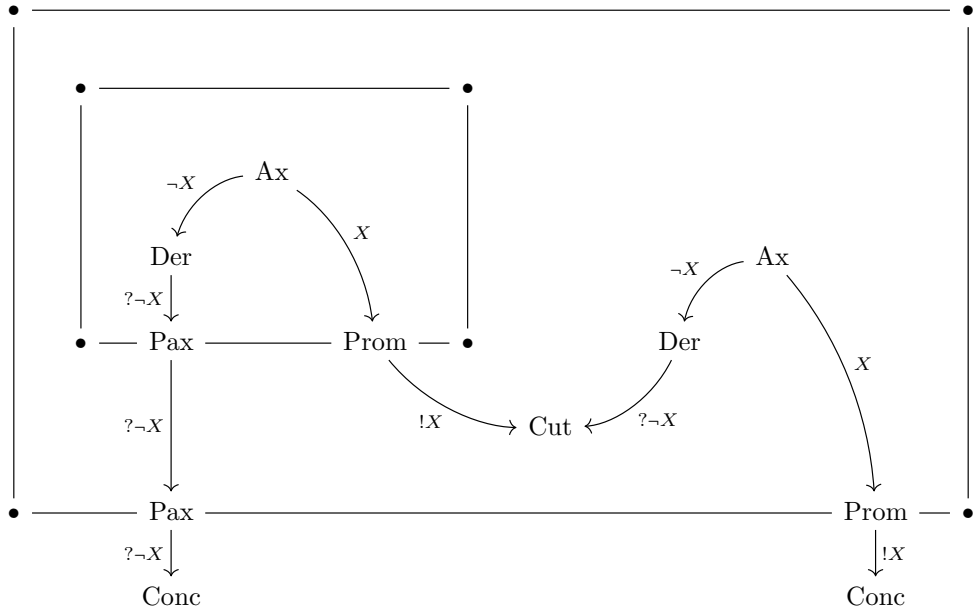
\[\begin{tikzcd}[column sep = small]
	\bullet &&&&&&&&&&& \bullet \\
	& \bullet &&&& \bullet \\
	&&& \ax \\
	&& \der &&&&&&& \ax \\
	& \bullet & \pax && \prom & \bullet &&& \der \\
	&&&&&& \cut \\
	\bullet && \pax &&&&&&&& \prom & \bullet \\
	&& \conc &&&&&&&& \conc
	\arrow[no head, from=1-1, to=7-1]
	\arrow[no head, from=1-12, to=1-1]
	\arrow[no head, from=2-2, to=5-2]
	\arrow[no head, from=2-6, to=2-2]
	\arrow[no head, from=2-6, to=5-6]
	\arrow["{\neg X}"', curve={height=12pt}, from=3-4, to=4-3]
	\arrow["X", curve={height=-12pt}, from=3-4, to=5-5]
	\arrow["{?\neg X}"', from=4-3, to=5-3]
	\arrow["{\neg X}"', curve={height=12pt}, from=4-10, to=5-9]
	\arrow["X", curve={height=-12pt}, from=4-10, to=7-11]
	\arrow[no head, from=5-2, to=5-3]
	\arrow[no head, from=5-3, to=5-5]
	\arrow["{?\neg X}"', from=5-3, to=7-3]
	\arrow[no head, from=5-5, to=5-6]
	\arrow["{!X}"', curve={height=12pt}, from=5-5, to=6-7]
	\arrow["{?\neg X}", curve={height=-12pt}, from=5-9, to=6-7]
	\arrow[no head, from=7-1, to=7-3]
	\arrow[no head, from=7-3, to=7-11]
	\arrow["{?\neg X}"', from=7-3, to=8-3]
	\arrow[no head, from=7-11, to=7-12]
	\arrow["{!X}", from=7-11, to=8-11]
	\arrow[no head, from=7-12, to=1-12]
\end{tikzcd}\]
It is cut-equivalent to the following proof net which is shallow:
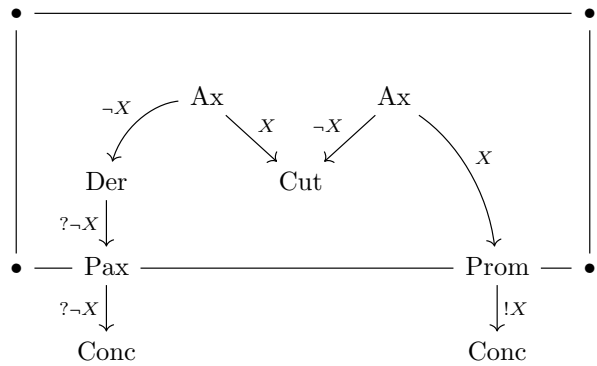
\[\begin{tikzcd}[column sep = small]
	\bullet &&&&&& \bullet \\
	&& \ax && \ax \\
	& \der && \cut \\
	\bullet & \pax &&&& \prom & \bullet \\
	& \conc &&&& \conc
	\arrow[no head, from=1-1, to=4-1]
	\arrow[no head, from=1-7, to=1-1]
	\arrow["{\neg X}"', curve={height=12pt}, from=2-3, to=3-2]
	\arrow["X", from=2-3, to=3-4]
	\arrow["{\neg X}"', from=2-5, to=3-4]
	\arrow["X", curve={height=-12pt}, from=2-5, to=4-6]
	\arrow["{?\neg X}"', from=3-2, to=4-2]
	\arrow[no head, from=4-1, to=4-2]
	\arrow[no head, from=4-2, to=4-6]
	\arrow["{?\neg X}"', from=4-2, to=5-2]
	\arrow[no head, from=4-6, to=4-7]
	\arrow["{!X}", from=4-6, to=5-6]
	\arrow[no head, from=4-7, to=1-7]
\end{tikzcd}\]

The following is not shallow because the linear part of the interior of the box fails to satisfy the property that all the persistent paths go through the premise of the Promotion-link:
\[\begin{tikzcd}[column sep = small]
	\bullet &&&&&& \bullet \\
	&& \ax && \ax \\
	& \der && \parr \\
	&&& \der \\
	\bullet & \pax && \pax && \prom & \bullet \\
	& \conc && \conc && \conc
	\arrow[no head, from=1-1, to=1-7]
	\arrow[no head, from=1-7, to=5-7]
	\arrow["{\neg X}"', curve={height=12pt}, from=2-3, to=3-2]
	\arrow["X", from=2-3, to=3-4]
	\arrow["{\neg X}"', from=2-5, to=3-4]
	\arrow["X", curve={height=-12pt}, from=2-5, to=5-6]
	\arrow["{?\neg X}"', from=3-2, to=5-2]
	\arrow["{X \parr \neg X}"', from=3-4, to=4-4]
	\arrow["{?(X \parr \neg X)}"', from=4-4, to=5-4]
	\arrow[no head, from=5-1, to=1-1]
	\arrow[no head, from=5-1, to=5-2]
	\arrow[no head, from=5-2, to=5-4]
	\arrow["{?\neg X}"', from=5-2, to=6-2]
	\arrow["{?(X \parr \neg X)}"', from=5-4, to=6-4]
	\arrow[no head, from=5-6, to=5-4]
	\arrow["{!X}", from=5-6, to=6-6]
	\arrow[no head, from=5-7, to=5-6]
\end{tikzcd}\]
\end{example}

\begin{remark}
The $\prom/\pax$-reduction step involves nesting a box within another, and so necessarily involves proof nets which are not shallow. As a result, we do not consider this reduction step inside this paper. In fact, it is possible to have a shallow proof net whose cut-elimination process necessarily involves a proof net which is \emph{not} shallow (Example \ref{eq:problematic} is such a proof net). Thus, strong normalisation does \emph{not} hold for the class of shallow proof nets. However, we \emph{do} obtain normalisation for many algorithms of interest, for instance addition of Church numerals admits normalisation where all proof nets involved are shallow proof nets.
\end{remark}

In Section \ref{sec:future} we give a research proposal for extending this model to all of MELL using a more general version of the Hilbert scheme.

\subsection{The projective schemes associated to the algebraic model}
In \cite{AlgPnt} we associated a coordinate ring $R_\pi$ to every MLL proof net $\pi$, defined as a quotient $R_\pi = P_\pi/I_\pi$, where $P_\pi$ is a polynomial ring and $I_\pi$ is an ideal. While this algebraic construction is suitable for MLL proof nets, it is more natural to consider the associated schemes for shallow proofs. This shift in perspective is particularly helpful when working with the Hilbert scheme, as the Hilbert scheme lacks a straightforward algebraic counterpart.

In this section we give an introduction to the Hilbert scheme, and in particular does not involve any novel content whatsoever.

\subsection{The Hilbert Functor}
The construction of the Hilbert scheme begins with the Hilbert \emph{functor}. Recall that for $R$ a ring, $M$ an $R$-module, and $k > 0$ an integer, $M$ is \textbf{locally free of rank $k$} if there exists $n>0$ and elements $f_1, \ldots, f_n \in R$ such that for all $i = 1 , \ldots, n$, $M_{f_i}$ is a free $R_{f_i}$-module of rank $k$. If $T$ is a graded $\mathbbm{k}$-algebra, and $h: \bb{N} \lto \bb{N}$ is a function, then the Hilbert functor of $T$ with respect to $h$ is a functor $\underline{H_T^h}: \mathbbm{k}-\underline{\operatorname{Alg}} \lto \underline{\operatorname{Set}}$ where $\underline{\operatorname{Set}}$ is the category of sets and functions. This functor maps a $\mathbbm{k}$-algebra $R$ to the following set, where $R \otimes T$ denotes $\bigoplus_{d \geq 0}R \otimes T_a$:
\begin{align*}
\underline{H_T^h}(R) = \{I \subseteq R \otimes T &\mid I\text{ is homogeneous and } \forall d \geq 0,\\
&(R \otimes T_d)/I_d\text{ is a locally free }R\text{-module }\\
&\text{of rank }h(d)\}.
\end{align*}
It was first proved by Grothendieck in \cite{HilbertScheme} that there exists a scheme $H_T^h$ representing this functor. That is, there is a natural isomorphism for each $R \in \mathbbm{k}-\underline{\operatorname{Alg}}$:
\begin{equation}
\underline{H_T^h}(R) \cong \operatorname{Hom}_{\underline{\operatorname{Sch}}_{\mathbbm{k}}}(\operatorname{Spec}R, H_T^h)
\end{equation}
where $\underline{\operatorname{Sch}}_{\mathbbm{k}}$ is the category of schemes $X \lto \operatorname{Spec}\mathbbm{k}$ over $\operatorname{Spec}\mathbbm{k}$ and morphisms of schemes commuting over $\operatorname{Spec}\mathbbm{k}$. We provide a detailed definition of this scheme in Section \ref{sec:Hilbert_scheme}, and its construction has been reproduced in \cite[Appendix D.6]{troiani2024phd}. This particular version of the Hilbert scheme along with its construction was first written down in \cite{hilb}.

\subsection{Properties of the Hilbert scheme}
Section \ref{sec:model} differs from the original Geometry of Interaction paper \cite{GoI} where rather than interpreting exponentials using the Hilbert \emph{scheme}, Girard interpreted exponentials using the Hilbert \emph{hotel}. To understand the geometry of our model one need not first acquire a knowledge of the Hilbert scheme's construction, but one must understand some of its properties. We have organised this paper so that the minimal theory of the Hilbert scheme required to understand our model is presented, and then the algebraic geometry involving the construction of the Hilbert scheme can be found in \cite{troiani2024phd}.

\subsubsection{The Grassmann scheme}
\label{sec:Grassmann}

\begin{defn}
\label{def:representable}
    Let $X$ be a scheme. We denote the following functor by $h_X$ which acts on objects as
    \begin{align*}
    h_X: \mathbbm{k}-\underline{\operatorname{Alg}} &\lto \underline{\operatorname{Set}}\\
    R &\longmapsto \operatorname{Hom}_{\underline{\operatorname{Sch}}_{\mathbbm{k}}}(\operatorname{Spec}R, X)
\end{align*}
and which maps a homomorphism of $\mathbbm{k}$-algebras $f: R \lto T$ to the composition map
\begin{align*}
    \hat{f} \circ (-): \operatorname{Hom}(\operatorname{Spec}R, X) &\lto \operatorname{Hom}(\operatorname{Spec}T, X)\\
    g &\longmapsto \hat{f} \circ g
\end{align*}
where $\hat{f}: \operatorname{Spec}T \lto \operatorname{Spec}R$ is induced by $f$.
\end{defn}

\begin{defn}
If $F: \mathbbm{k}-\underline{\operatorname{Alg}} \lto \underline{\operatorname{Set}}$ is a functor and there exists a scheme $X$ such that $F \cong h_X$, then $F$ is \textbf{representable} and is \textbf{represented} by $X$.
\end{defn}

Let $R$ be a $\mathbbm{k}$-algebra. Let $n > 0$, $0 < k < n$ and define the following set:
    \begin{align*}
        \underline{G_n^k}(R) = \{ L \subseteq R^n &\mid L\text{ is an }R\text{ submodule, and }\\
        &R^n/L\text{ is a locally free }R\text{-module of rank }k\}.
    \end{align*}
Given an element $L \in \underline{G_n^k}(R)$ and a $\mathbbm{k}$-algebra homomorphism $\phi: R \lto S$ we can tensor the short exact sequence
\begin{equation}
\begin{tikzcd}
0\arrow[r] & L\arrow[r] & R^n\arrow[r] & R^n/L\arrow[r] & 0
\end{tikzcd}
\end{equation}
by $S$ over $R$ and obtain a new short exact sequence which is isomorphic to the following:
\[\begin{tikzcd}
    0 & {S \otimes_{R}L} & {S^n} & {S^n/(S \otimes_{R}L)} & 0.
    \arrow[from=1-1, to=1-2]
    \arrow[from=1-2, to=1-3]
    \arrow[from=1-3, to=1-4]
    \arrow[from=1-4, to=1-5]
\end{tikzcd}\]
It follows that $S^n/(S \otimes_{\mathbbm{k}}L)$ is locally free of rank $k$ if $R^n/L$ is. Thus we have a well defined map $\underline{G_n^k}(R) \lto \underline{G_n^k}(S): L \longmapsto S \otimes_R L$ which is denoted $\underline{G_n^k}(\phi)$. This extends $\underline{G_n^k}$ to a functor.

\begin{defn}
\label{def:Grassmann_functor}
    The functor $\underline{G_n^k}: \mathbbm{k}-\underline{\operatorname{Alg}} \lto \underline{\operatorname{Set}}$ is the \textbf{Grassmann Functor}.
\end{defn}

\begin{example}
Consider the $\bb{C}$-algebra $\bb{C}^2$. Then for any $\bb{C}$-algebra $R$ we have $R \otimes_{\bb{C}}\bb{C}^2 \cong R^2$. Let $e_1, e_2$ be the standard $R$-basis for $R^2$ and consider a short exact sequence
\[\begin{tikzcd}
    0 & {\operatorname{Span}_R\{e_1 - e_2\}} & {R^2} & R & 0
    \arrow[from=1-1, to=1-2]
    \arrow[from=1-2, to=1-3]
    \arrow[from=1-3, to=1-4]
    \arrow[from=1-4, to=1-5]
\end{tikzcd}\]
then $\operatorname{Span}_R\{e_1 - e_2\} \in \underline{G_2^1}(R)$.
\end{example}

Let $\{ e_1, \ldots, e_n \}$ be the standard basis vectors for $R^n$ and $B = \{e_{i_1}, \ldots, e_{i_k}\} \subseteq \{ e_1, \ldots, e_n \}$ be a size $k$ subset with $i_1 < \ldots < i_k$. Among the elements of $\underline{G_n^k}(R)$ are the modules $L \in \underline{G_n^k}(R)$ such that $R^n/L$ has basis $\{ [e_{i_1}]_L, \ldots, [e_{i_k}]_L \}$, where for $i_1, \ldots, i_k$ the notation $[e_{i_j}]_L$ denotes the image of $e_i \in R^n$ under the standard quotient map $R^n \lto R^n/L$. We will denote by $[B]_{L}$ the set $\{ [e_{i_1}]_L, \ldots, [e_{i_k}]_L \}$.

\begin{defn}
    Define the following subset
    \begin{equation}
    \underline{G^k_{n\backslash B}}(R) := \{ L \in \underline{G^k_n}(R) \mid R^n/L\text{ is free with }R\text{-basis }[B]_L \} \subseteq \underline{G^k_n}(R).
\end{equation}
This extends to a full subfunctor of $\underline{G_n^k}$.
\end{defn}
\begin{lemma}
\label{lem:Grassmann_local_rep}
    The functor $\underline{G^k_{n\backslash B}}$ is represented by
    \begin{equation}
    \label{eq:representing_scheme}
        \operatorname{Spec}\mathbbm{k}[\{z_i^j \mid 1 \leq i \leq k, 1 \leq j \leq n - k\}].
    \end{equation}
\end{lemma}
\begin{proof}
    Fix a $\mathbbm{k}$-algebra $R$. If $L \in \underline{G_{n\backslash B}^k}$ then for each $e_{m} \not\in B$ we have
    \begin{equation}
        [e_{m}] = \sum_{i = 1}^k \alpha_{i}^j[e_{i_j}]
    \end{equation}
    for some coefficients $\{\alpha_{i}^j\}_{1 \leq i \leq k, 1 \leq j \leq n-k} \subseteq R$. The data of these coefficients is equivalent to the data of a $\mathbbm{k}$-algebra morphism
    \begin{equation}
        \mathbbm{k}[\{z_i^j\}] \lto R
    \end{equation}
    which in turn is equivalent to the data of a morphism $\operatorname{Spec}R \lto \operatorname{Spec}\mathbbm{k}[\{z_i^j\}]$.
\end{proof}

\begin{proposition}
\label{prop:Grassmann_representable}
For all $n > k > 0$, the functor $\underline{G_n^k}$ is represented by a closed subscheme of $\bb{P}^{\binom{n}{k}-1}$.
\end{proposition}
\begin{proof}
The representing scheme can be constructed by gluing together the schemes \ref{eq:representing_scheme} ranging over all $B$, see \cite[Appendix D.5]{troiani2024phd}.
\end{proof}

\begin{defn}
\label{def:Grassmann_embedding}
We denote the projective scheme representing the functor $\underline{G_n^k}$ by $G_n^k$. This is the \textbf{Grassmann scheme}.
\end{defn}

\subsubsection{The Hilbert scheme}
\label{sec:Hilbert_scheme}
We follow \cite{hilb}.
\begin{defn}
    A \textbf{graded $\mathbbm{k}$-module with operators} is a pair $(T,F)$ consisting of a graded $\mathbbm{k}$-module
    \begin{equation}
        T = \bigoplus_{d \in \bb{N}}T_d
    \end{equation}
    and a family of operators
    \begin{equation}
        F = \bigcup_{d,e \in \bb{N}}F_{d,e}
    \end{equation}
    where for all $d,e \in \bb{N}, F_{d,e} \subseteq \operatorname{Hom}(T_d, T_e)$.
\end{defn}
\begin{defn}
    Let $(T,F)$ be a graded $\mathbbm{k}$-module with operators. A graded submodule
    \begin{equation}
        L = \bigoplus_{d \in \bb{N}}L_d \subseteq T
    \end{equation}
    is an \textbf{$F$-submodule} if $F_{d,e}(L_d) \subseteq L_e$ for all $d, e \in \bb{N}$.
\end{defn}
\begin{example}
\label{ex:algebra_F_module}
    If $T$ is a graded $\mathbbm{k}$-algebra then for $a,b \in \bb{N}$, define:
    \begin{equation}
        F_{a,b} = 
        \begin{cases}
            \text{The set of multiplications by monomials of degree }e - d,& d \geq e\\
            \varnothing,&d<e
        \end{cases}
    \end{equation}
    then any homogeneous ideal is a homogeneous $F$-module where $F = \{F_{d,e}\}_{d,e \in \bb{N}}$.
\end{example}

\begin{defn}
If $(T,F)$ is a graded $\mathbbm{k}$-module with operators and $D \subseteq \bb{N}$ is a subset of the degrees, we denote by $(T_D, F_D)$ the graded $\mathbbm{k}$-module with operators where
\begin{equation}
T_D = \bigoplus_{d \in D}T_d,\quad F_D = \{F_{d,e} \in F\mid d,e \in D\operatorname\}.
\end{equation}
\end{defn}

Let $R$ be a commutative $\mathbbm{k}$-algebra. Notice that if $(T,F)$ is a graded $\mathbbm{k}$-module with operators, then so is
\begin{equation}
    R \otimes T := \bigoplus_{d \in \bb{N}}R \otimes T_d
\end{equation}
when paired with the operators $\hat{F} = \{\operatorname{id}_R \otimes F_{d,e}\}_{d,e \in \bb{N}}$. Given a function $h: \bb{N} \lto \bb{N}$ we define the set
\begin{align*}
    \underline{H_T^h}(R) = \{ F-\text{submodules } L \subseteq R \otimes T &\mid \forall d \in \bb{N}, (R \otimes T_d)/L_d\text{ is }\\
    &\text{ locally free of rank }h(d)\}.
\end{align*}
Let $\phi: R \lto S$ be a $\mathbbm{k}$-algebra homomorphism and let $f_1, \ldots, f_n \in R$ be a set of elements generating the unit ideal. Then for any $d \in \bb{N}$ and any $i = 1, \ldots, n$ there is a short exact sequence
\begin{equation}
\begin{tikzcd}
0\arrow[r] & (L_d)_{f_i}\arrow[r] & (R \otimes T_d)_{f_i}\arrow[r] & (R\otimes T_d/L_d)_{f_i}\arrow[r] & 0.
\end{tikzcd}
\end{equation}
By tensoring with $S$ over $R$ we obtain a similar short exact sequence. The function $\underline{H_T^h}(R) \lto \underline{H_T^h}(S), L \lto S \otimes L$ is denoted $\underline{H_T^h}(\phi)$. It is easy to see that $\underline{H_T^h}: \mathbbm{k}-\underline{\operatorname{Alg}} \lto \underline{\operatorname{Set}}$ is a functor.
\begin{defn}
    The functor $\underline{H_T^h}$ is the \textbf{Hilbert functor}.
\end{defn}
\begin{defn}
Let $D \subseteq \bb{N}$. The \textbf{restriction} is the following natural transformation $\operatorname{Res}_{T_D}: \underline{H_T^h} \lto \underline{H_{T_D}^h}$ which maps an element $L \in \underline{H_T^h}(R)$ to the restriction $L_D = \bigoplus_{d \in D}L_d$.
\end{defn}

\begin{thm}
\label{thm:Hilbert_scheme}
    Let $(T,F)$ be a graded $\mathbbm{k}$-module with operators. Let $h: \bb{N} \lto \bb{N}$ be a function such that $\sum_{d \in \bb{N}}h(d) < \infty$. Suppose $M \subseteq N \subseteq T$ are homogeneous $\mathbbm{k}$-submodules satisfying:
    \begin{itemize}
        \item $N$ is a finitely generated $\mathbbm{k}$-module.
        \item $N$ generates $T$ as an $F$-module.
        \item For every field $K \in \mathbbm{k}-\underline{\operatorname{Alg}}$ and every $L \in \underline{H_T^h(K)}$, $M$ generates $(K \otimes T)/L$ as a $K$-module.
        \item There is a subset $G \subseteq F$ so that $G$ is the closure of $F$ under composition and $G$ is such that $GM \subseteq N$.
    \end{itemize}
    Then $\underline{H_T^h}$ is represented by a quasiprojective scheme $H_T^h$.
\end{thm}
\begin{proof}
See \cite[Theorem 2.2]{hilb}. We also reproduced this proof in \cite[Section D.6]{troiani2024phd}.
\end{proof}

Theorem \ref{thm:Hilbert_scheme} only holds when $h$ is such that $\sum_{d \in \bb{N}}h(d) < \infty$ because we construct $H_T^h$ as a subscheme of $G^r_n$ for some $r > \sum_{d \in \bb{N}}h(d)$. We wish to apply Theorem \ref{thm:Hilbert_scheme} in the setting where $h$ is the Hilbert function (recalled in Definition \ref{def:Hilbert_function}) of a homogeneous ideal of $\mathbbm{k}[x_0, \ldots, x_n]$ (given with respect to the standard grading). This function in general is \emph{not} of finite support. To mitigate this, we follow \cite{hilb} and construct a subset $D \subseteq \bb{N}$ to exhibit the Hilbert functor $\underline{H_T^h}$ as a subfunctor of $\underline{H_{T_D}^h}$. We then relate to this a closed immersion of schemes $H_T^h \lto H_{T_D}^h$.

\begin{defn}
    \label{def:Hilbert_function}
    Let $S$ be a graded $\mathbbm{k}$-algebra and $I \subseteq S$ a homogeneous ideal. The \textbf{Hilbert function} of $I$ is the function
    \begin{align*}
        \bb{N} &\lto \bb{N}\\
        n &\longmapsto \operatorname{dim}_{\mathbbm{k}}(S_n/I_n).
    \end{align*}
\end{defn}

\begin{proposition}
\label{prop:unique_expansion}
Let $d > 0, c > 0$. There exists a unique expression
\begin{equation}
\label{eq:unique_expansion}
c = \binom{k_d}{d} + \binom{k_{d-1}}{d-1} + \ldots + \binom{k_\delta}{\delta}
\end{equation}
where $k_d > k_{d-1} > \ldots > k_\delta \geq \delta > 0$.
\end{proposition}
\begin{proof}
We proceed by induction on $c$.

Say $c = 1$. Then for any $d > 0$
\begin{equation}
\label{eq:1}
    1 = \binom{d}{d}.
\end{equation}
Since for all $a>d$ we have $\binom{a}{d} > 1$ it is clear that \eqref{eq:1} is the unique such expression.

Say $c > 1$. First we prove existence of such an expression. Let $k_d$ denote the largest integer such that
\begin{equation}
\label{eq:greater_than}
    c \geq \binom{k_d}{d}.
\end{equation}
If \eqref{eq:greater_than} holds to equality then we are done, so assume $c - \binom{k_d}{d} > 0$ which by the inductive hypothesis implies there exists unique $k_{d-1} > \ldots > k_{\delta} > 0$ such that
\begin{equation}
\label{eq:inductive}
    c - \binom{k_d}{d} = \binom{k_{d-1}}{d-1} + \ldots + \binom{k_{\delta}}{\delta}.
\end{equation}
We must show that $k_d > k_{d-1}$. Suppose to the contrary that $k_d \leq k_{d-1}$. Then
\begin{equation}
    \binom{k_d}{d-1} \leq \binom{k_{d-1}}{d-1}
\end{equation}
and so using \eqref{eq:inductive} we have
\begin{equation}
    c \geq \binom{k_{d-1}}{d-1} + \binom{k_d}{d} \geq \binom{k_{d}}{d-1} + \binom{k_d}{d} = \binom{k_{d}+1}{d}
\end{equation}
which contradicts maximality of $k_d$.

Now we prove uniqueness. Assume there were two expressions:
\begin{align*}
    c &= \binom{k_d}{d} + \binom{k_{d-1}}{d-1} + \ldots + \binom{k_\delta}{\delta}\\
    c &= \binom{k_d'}{d} + \binom{k_{d-1}'}{d-1} + \ldots + \binom{k_{\delta'}'}{\delta'}
\end{align*}
with $k_d > k_{d-1} > \ldots > k_\delta > 0, k_d' > k_{d-1}' > \ldots > k_{\delta'}' > 0$. Let $s \leq d$ be the greatest integer such that $k_s \neq k_s'$. By considering $c - \sum_{i = s}^d \binom{k_{i}}{i}$ in place of $c$ we may assume $s = d$.

Assume without loss of generality that $k_d' < k_d$. Since $k_d'$ is an integer we have $k_d' \leq k_d - 1$. By the inductive hypothesis, the expression
\begin{equation}
    c - \binom{k_d'}{d} = \binom{k_{d-1}'}{d-1} + \ldots + \binom{k_{\delta'}'}{\delta'}
\end{equation}
is the unique such, and so $k_{d-1}'$ is the maximal integer such that
\begin{equation}
    c - \binom{k_d'}{d} \geq \binom{k_{d-1}'}{d-1}.
\end{equation}
Since $k_d > k_d'$, we have:
\begin{equation}
    c - \binom{k_d'}{d} > c - \binom{k_d}{d} = \binom{k_{d}-1}{d-1} + \ldots + \binom{k_\delta}{\delta} \geq \binom{k_d-1}{d-1}
\end{equation}
and so $k_d - 1 \leq k_{d-1}'$. Thus, $k_d' \leq k_{d-1}'$, which is a contradiction.
\end{proof}

\begin{defn}
The \textbf{$d$-binomial expansion of $c$} is the unique expansion given by \eqref{eq:unique_expansion}. The \textbf{$d^{\text{th}}$ Macaulay difference set of $c$}, $M_d(c)$ is defined as the tuple
\begin{equation}
M_d(c) = (k_d - d, d_{d-1} - (d-1), \ldots, k_\delta - \delta).
\end{equation}
\end{defn}

We note that the data of the $d$-binomial expansion of $c$ is equivalent to that of the $d^{\text{th}}$ Macaulay difference set of $c$.

\begin{example}
The following is the 4-binomial expansion of $27$:
\begin{equation}
27 = \binom{6}{4} + \binom{5}{3} + \binom{2}{2} + \binom{1}{1}.
\end{equation}
The $4^{\text{th}}$ Macaulay difference set of 27 is $(2,2,0,0)$.
\end{example}

\begin{defn}
Let $c > 0, d > 0$, and let $k_d > k_{d-1} > \ldots > k_\delta \geq \delta > 0$ be the integers involved in the $d$-binomial expansion of $c$ as in Proposition \ref{prop:unique_expansion}. Define the following natural number:
\begin{equation}
c^{\langle d \rangle} = \binom{k_d + 1}{d+ 1} + \binom{k_{d-1} + 1}{d} + \ldots + \binom{k_{\delta} + 1}{\delta + 1}.
\end{equation}
\end{defn}

\begin{remark}
\label{rmk:obvious}
The $d^{\text{th}}$ Macaulay difference set of $c$ and the $(d+1)^{\text{th}}$ Macaulay difference set of $c^{\langle d \rangle}$ are equal.
\end{remark}

\begin{proposition}
\label{prop:ladder}
Fix $n > 0$ and a homogeneous ideal $I \subseteq \mathbbm{k}[x_0, \ldots, x_n]$. Let $h$ be the Hilbert function of $I$. There exists an integer $j$ such that for all $d \geq j$ we have
\begin{equation}
\label{eq:ladder}
h(d+1) = h(d)^{\langle d \rangle}.
\end{equation}
\end{proposition}
\begin{proof}
See \cite[Section 2]{Gotzmann}.
\end{proof}
\begin{cor}
Let $I \subseteq \mathbbm{k}[x_0, \ldots, x_n]$ be homogeneous with Hilbert function $h$. Let $j$ be the integer such that for all $d \geq j$ we have \eqref{eq:ladder}. Then for all $d \geq j$ the $d^{\text{th}}$ Macaulay difference set of $h(d)$ is equal to the $j^{\text{th}}$ Macaulay difference set of $h(j)$.
\end{cor}
\begin{proof}
For $d \geq j$ we have
\begin{equation}
    M_{d+1}(h(d+1)) = M_{d+1}(h(d)^{\langle d \rangle}) = M_{d}(h(d))
\end{equation}
where the first equality holds by Proposition \ref{prop:ladder} and the second by Remark \ref{rmk:obvious}.
\end{proof}

\begin{defn}
Let $I \subseteq \mathbbm{k}[x_0, \ldots, x_n]$ be a homogeneous ideal. The \textbf{Gotzmann number $G(I)$ of $I \subseteq \mathbbm{k}[x_0, \ldots, x_n]$} is the number of elements in the eventually constant $d^{\text{th}}$ Macaulay difference set of $h(d)$.
\end{defn}

\begin{example}
\label{ex:Gotmann_number_ex}
Consider the Segre embedding (see Corollary \cite[Corollary 3.7]{troiani2024phd} for a reminder) $\operatorname{Seg}: \bb{P}^1 \times \bb{P}^1 \lto \bb{P}^3$ and the canonical closed immersion of the diagonal $\iota: \Delta \lto \bb{P}^1 \times \bb{P}^1$. Since these are both closed immersions, so is their composite $\operatorname{Seg}\iota: \Delta \lto \bb{P}^3$. The image of this closed immersion corresponds uniquely to a saturated homogeneous ideal $I \subseteq S = \mathbbm{k}[Z_{00}, Z_{01}, Z_{10}, Z_{11}]$. This ideal $I$ is
\begin{equation}
I = (Z_{01} - Z_{10}, Z_{00}Z_{11} - Z_{01}Z_{10}).
\end{equation}
We calculate the Gotzmann number of $I \subseteq S$. First we calculate the Hilbert function. We can calculate the Hilbert function of $I \subseteq S$ directly using a minimal free graded resolution of $S/I$. Let $f = Z_{01} - Z_{10}, g = Z_{00}Z_{11} - Z_{01}Z_{10}$. Then we have the following minimal graded free resolution, where for $d > 0$ the notation $S(d)$ denotes the graded $\mathbbm{k}$-algebra $S$ with degree shifted by $d$:
\[\begin{tikzcd}[column sep = large, ampersand replacement = \&]
    0 \& {S(-3)} \& {S(-1) \oplus S(-2)} \& S \& {S/I} \& 0.
    \arrow[from=1-1, to=1-2]
    \arrow["{\begin{pmatrix}g\\-f\end{pmatrix}}", from=1-2, to=1-3]
    \arrow["{\begin{pmatrix}f&g\end{pmatrix}}", from=1-3, to=1-4]
    \arrow[from=1-4, to=1-5]
    \arrow[from=1-5, to=1-6]
\end{tikzcd}\]
Thus for any $d \geq 0$:
\begin{align*}
0 &= \operatorname{dim}S(-3)_d - \operatorname{dim}S(-1)_d - \operatorname{dim}S(-2)_d + \operatorname{dim}S_d - \operatorname{dim}(S/I)_d\\
&= \operatorname{dim}S_{d-3} - \operatorname{dim}S_{d-1} - \operatorname{dim}S_{d-2} + \operatorname{dim}S_d - \operatorname{dim}(S/I)_d.
\end{align*}
In general, if $S' = \mathbbm{k}[x_1, \ldots, x_n]$ then the dimension of $S'_d$ is the number of monomials in $n$ variables of degree $d$. This number is
\begin{equation}
\operatorname{dim}S'_d = \binom{n + d - 1}{d}.
\end{equation}
Here, $n = 4$, so:
\begin{equation}
\operatorname{dim}(S/I)_d = \binom{d}{d-3} - \binom{d+2}{d-1} - \binom{d+1}{d-2} + \binom{d+3}{d}
\end{equation}
which is equal to $2d+1$. So, the Hilbert function of $I$ is $h: \bb{N} \lto \bb{N}, h(d) = 2d+1$. Notice that
\begin{equation}
2d + 1 = \binom{d + 1}{d} + \binom{d}{d-1}.
\end{equation}
By uniqueness of such expressions (Proposition \ref{prop:unique_expansion}) it follows that the Macaulay difference set is $(1,1)$ and the Gotzmann number $G(I)$ of $I$ is 2.
\end{example}

\begin{defn}
Let $D \subseteq \bb{N}$. We say that $D$ is \textbf{supportive} if the canonical morphism $H_S^h \lto H_{S_D}^h$ is a closed immersion. It is $\textbf{very supportive}$ if $H_S^h \lto H_{S_D}^h$ is an isomorphism (see \cite[Corollary 3.4]{hilb}).
\end{defn}

For the remainder of this Section let $S = \mathbbm{k}[x_0, \ldots, x_n]$ for some fixed $n > 0$.

\begin{proposition}
\label{prop:big_assumption}
Let $I \subseteq S$ be a homogeneous ideal with Hilbert function $h$. Let $G(I)$ denote the Gotzmann number of $I \subseteq S$. Then the set $\{G(I)\}$ is supportive and the set $\{G(I),G(I)+1\}$ is very supportive.
\end{proposition}
\begin{proof}
See \cite[Proposition 4.2]{hilb}.
\end{proof}

\begin{cor}
\label{cor:Grothendieck_embedding}
Let $I \subseteq S$ be a homogeneous ideal with Hilbert function $h$. Let $G(I)$ denote the Gotzmann number of $I \subseteq S$ and let $D = \{G(I)\}$. Denote by $r,s$ the following integers
\begin{equation}
r = \binom{n + G(I) - 1}{G(I)},\quad s = \binom{r}{h(G(I))}.
\end{equation}
Then there exists a sequence of closed immersions
\begin{equation}
\label{eq:big_embedding}
H_S^h \lto H_{S_{D}}^h \lto G_{r}^{h(G(I))} \lto \bb{P}^{s-1}.
\end{equation}
In particular, $H_S^h$ is projective.
\end{cor}

In Section \ref{sec:model} we will need to fixed a choice of closed immersion of the Hilbert scheme $H_S^h$ into projective space $\bb{P}^{s-1}$, for each polynomial ring $S = \mathbbm{k}[x_0, \ldots, x_n]$ and Hilbert function $h: \bb{N} \lto \bb{N}$, where $s$ is as defined in Corollary \ref{cor:Grothendieck_embedding}. We fix once and for all such a choice and refer to this as the \textbf{Grothendieck immersion}.

\begin{remark}
\label{rmk:equations}
We only consider shallow proofs in this paper, for which the details of the immersion \eqref{eq:big_embedding} are not necessary, though we will use that $H_S^h$ is projective. In order to extend the model of Section \ref{sec:model} to all of MELL it seems necessary to prove certain properties of at least one of the sets of equations which define an ideal $I$ such that $\operatorname{Proj}(S/I) \cong H_S^h$.
\end{remark}

\section{Exponentials}
\label{sec:model}

\begin{defn}
\label{def:Hilbert_functions_set}
Let $\mathcal{H}$ denote the set of all Hilbert functions $h: \bb{N} \lto \bb{N}$.
\end{defn}

\begin{defn}
\label{def:formula_schemes}
Let $A$ be a shallow formula. The \textbf{scheme of $A$}, $\bb{S}(A)$, is defined inductively to be a disjoint union of projective spaces as follows:
\begin{itemize}
	\item Say $A = (X,x)$ is atomic. Then $\bb{S}(A) = \bb{P}^1$.
	\item Say $A = A_1 \otimes A_2$ and $\bb{S}(A_1) = \coprod_{i \in I} \bb{P}^{r_i}, \bb{S}(A_2) = \coprod_{j \in J} \bb{P}^{s_j}$. Recall that for each pair $(i,j) \in I \times J$ there is the Segre embedding: $\bb{P}^{r_i} \times \bb{P}^{s_j} \lto \bb{P}^{(r_i+1)(s_j+1)-1}$, see \cite[Corollary 3.7]{troiani2024phd} for a reminder. Define
	\begin{equation}
		\bb{S}(A) = \coprod_{i \in I}\coprod_{j \in J}\bb{P}^{(r_i + 1)(s_j + 1) - 1}.
	\end{equation}
	\item Say $A = !B$ with $A$ linear. Recall that for each $h \in \mathcal{H}$ we have the Grothendieck immersion $H_S^h \lto \bb{P}^{s_h}$, for some integer $s_h$. Define
	\begin{equation}
		\bb{S}(A) = \coprod_{h \in \mathcal{H}} \bb{P}^{s_h}.
	\end{equation}
\end{itemize}
\end{defn}

\begin{defn}
    Let $e$ be an edge in a proof net. We denote by $A_e$ the formula labelling $e$.
\end{defn}

\begin{defn}
\label{def:MELL_interpretation}
    The \textbf{ambient scheme of $\pi$}, denoted $\bb{S}(\pi)$, is the product of all schemes of formulas ranging over all edges $e$ in $\pi$. That is, let $\mathcal{E}_\pi$ denote the set of edges of $\pi$ then 
    \begin{equation}
        \bb{S}(\pi) = \prod_{e \in \mathcal{E}_\pi} \bb{S}(A_e).
    \end{equation}
\end{defn}

We now define for each shallow proof $\pi$ an associated scheme $\bb{X}(\pi)$, along with a morphism of schemes $\iota_\pi: \bb{X}(\pi) \lto \bb{S}(\pi)$ which when restricted to any connected component of $\bb{X}(\pi)$ is a closed immersion. The scheme $\bb{X}(\pi)$ will be defined by associating to each link $l$ of $\pi$ a set of edges $\mathcal{L}_l$ of $\pi$ and a locally closed subscheme $\bb{X}(l)$ of $\prod_{e \in \mathcal{L}_l}\bb{S}(A_e)$.

\begin{defn}
\label{def:proof_interpretation}
For every pair of formulas $A,B$, write $\bb{S}(A) = \coprod_{i \in I}\bb{P}^{r_i}, \bb{S}(B) = \coprod_{j \in J}\bb{P}^{s_j}$ and fix an isomorphism
\begin{equation}
\phi_{\operatorname{M}}: \bb{S}(A) \times \bb{S}(B) \lto \coprod_{i \in I}\coprod_{j \in J}(\bb{P}^{r_i} \times \bb{P}^{s_j}).
\end{equation}

For any Hilbert function $h \in \mathcal{H}$ let $s_{h} > 0$ be such that $\bb{S}(?A) = \coprod_{h \in \mathcal{H}}\bb{P}^{s_h}$, fix an isomorphism
\begin{equation}
\phi_{\operatorname{D}}: \bb{S}(?A) \times \bb{S}(A) \lto \coprod_{h \in \mathcal{H}}\big(\bb{P}^{s_{h}} \times \bb{S}(A)\big).
\end{equation}

For every sequence $i = 1, \ldots, n$, every set of formulas $?A_1, \ldots, ?A_n$, with $\bb{S}(?A_i) = \coprod_{h_i \in \mathcal{H}}\bb{P}^{s_{h_i}}$, and every linear formula $B$ we fix an isomorphism
\begin{equation}
\label{eq:first_promotion_iso}
\phi_{\operatorname{P}^1}: \prod_{i = 1}^n \bb{S}(?A_i) \times \bb{S}(B) \lto \coprod_{\mathbf{h} \in \mathcal{H}^n}\prod_{i=1}^n\big(\bb{P}^{s_{h_i}} \times \bb{S}(B)\big).
\end{equation}
For $\bb{S}(!B) = \coprod_{h \in \mathcal{H}}\bb{P}^{s_h}$ we fix an isomorphism
\begin{equation}
\phi_{\operatorname{P}^2}: \prod_{i = 1}^n \bb{S}(?A_i) \times \bb{S}(!B) \lto \coprod_{\mathbf{h} \in \mathcal{H}^n}\coprod_{h \in \mathcal{H}}\prod_{i=1}^n\big(\bb{P}^{s_{h_i}} \times \bb{P}^{s_h}\big).
\end{equation}
The $\operatorname{M}, \operatorname{D}, \operatorname{P}$ stand respectively for ``Multiplicative", ``Dereliction", and ``Promotion".

Let $l$ be a link of a shallow proof net $\pi$. If $l$ is not a Promotion-link then let $\mathcal{L}_l$ denote the set of edges incident to $l$. If $l$ is a Promotion-link then let $\mathcal{L}_l$ denote the set of edges which are conclusions to the Promotion-link and all associated Pax-links. We define a closed subscheme $\bb{X}(l)$ of $\prod_{e \in \mathcal{L}_l}\bb{S}(A_e)$ along with a morphism
\begin{equation}
\iota_l: \bb{X}(l) \lto \prod_{e \in \mathcal{L}_l}\bb{S}(A_e).
\end{equation}

\textbf{Conclusion-link}
\[\begin{tikzcd}
	\vdots \\
	{\operatorname{c}}
	\arrow["A"', from=1-1, to=2-1]
\end{tikzcd}\]
We define $\bb{X}(l)$ to be the full subscheme $\bb{S}(A)$ of $\bb{S}(A)$ and take $\iota_l$ to be the identity morphism
\begin{equation}
\iota_{l}: \bb{X}(l) = \bb{S}(A) \lto \bb{S}(A).
\end{equation}

\textbf{Axiom- or Cut-link.}
\[\begin{tikzcd}[column sep = small]
	& {\operatorname{Ax}} && \vdots && \vdots \\
	\vdots && \vdots && {\operatorname{Cut}}
	\arrow["{\neg A}"', curve={height=12pt}, from=1-2, to=2-1]
	\arrow["A", curve={height=-12pt}, from=1-2, to=2-3]
	\arrow["{\neg A}"', curve={height=12pt}, from=1-4, to=2-5]
	\arrow["A", curve={height=-12pt}, from=1-6, to=2-5]
\end{tikzcd}\]
In both cases, we use the fact that $\bb{S}(\neg A) = \bb{S}(A)$. We define $\bb{X}(l)$ to be the diagonal $\Delta_{\bb{S}(A)}$ and define $\iota_l$ to be the canonical morphism
\begin{equation}
\iota_l: \bb{X}(l) = \Delta_{\bb{S}(A)} \lto \bb{S}(\neg A) \times \bb{S}(A).
\end{equation}

\textbf{Tensor- or Par-link}.
\[\begin{tikzcd}[column sep = small]
	\vdots && \vdots & \vdots && \vdots \\
	& \otimes &&& \parr \\
	& \vdots &&& \vdots
	\arrow["A"', curve={height=12pt}, from=1-1, to=2-2]
	\arrow["B", curve={height=-12pt}, from=1-3, to=2-2]
	\arrow["{A\otimes B}", from=2-2, to=3-2]
	\arrow["A"', curve={height=12pt}, from=1-4, to=2-5]
	\arrow["B", curve={height=-12pt}, from=1-6, to=2-5]
	\arrow["{A \parr B}", from=2-5, to=3-5]
\end{tikzcd}\]
Let $\boxtimes \in \{\otimes, \parr\}$. Write $\bb{S}(A) = \coprod_{i \in I}\bb{P}^{r_i}, \bb{S}(B) = \coprod_{j \in J}\bb{P}^{s_j}$. For each pair $(i,j) \in I \times J$ there exists the Segre embedding
\begin{equation}
\bb{P}^{r_i} \times \bb{P}^{s_j} \lto \bb{P}^{(r_i + 1)(s_j + 1) - 1}.
\end{equation}
We compose with the canonical inclusion morphism to obtain
\begin{equation}
\bb{P}^{r_i} \times \bb{P}^{s_j} \lto \coprod_{i \in I}\coprod_{j \in J}\bb{P}^{(r_i + 1)(s_j + 1) - 1} = \bb{S}(A \boxtimes B).
\end{equation}
By the universal property of the coproduct this induces a morphism
\begin{equation}
\coprod_{i \in I}\coprod_{j \in J}\bb{P}^{r_i} \times \bb{P}^{s_j} \lto \bb{S}(A \boxtimes B)
\end{equation}
which we pre-compose with $\phi_{\operatorname{M}}^{-1}$ to obtain
\begin{equation}
f: \bb{S}(A) \times \bb{S}(B) \lto \bb{S}(A \boxtimes B).
\end{equation}
We take the graph $\Gamma_f$ of $f$ to be $\bb{X}(l)$ and the canonical inclusion to be $\iota_l$:
\begin{equation}
\iota_l: \bb{X}(l) = \Gamma_f \lto \bb{S}(A) \times \bb{S}(B) \times \bb{S}(A \boxtimes B).
\end{equation}

\textbf{Dereliction-link}.
\[\begin{tikzcd}
	\vdots \\
	{?} \\
	\vdots
	\arrow["A"', from=1-1, to=2-1]
	\arrow["{? A}"', from=2-1, to=3-1]
\end{tikzcd}
\]
We have assumed that $\pi$ is shallow and so $A$ is linear. Thus if $m$ denotes the number of unoriented atoms of $A$ then $\bb{S}(A) = \bb{P}^{2^m - 1}$. Let $S$ denote the graded $\mathbbm{k}$-algebra $\mathbbm{k}[x_0, \ldots, x_{2^m - 1}]$. For each $h \in \mathcal{H}$ there exists an integer $s_h$ such that $\bb{S}(?A) = \coprod_{h \in \mathcal{H}}\bb{P}^{s_h}$. Fix $h \in \mathcal{H}$. Let $U = \operatorname{Spec}R$ denote an open affine of the Hilbert scheme $H_{S}^h \subseteq \bb{P}^{r_s}$, and consider the bijection
\begin{equation}
\psi: \underline{H_{S}^h}(R) \cong \operatorname{Hom}_{\underline{\operatorname{Sch}}_{\mathbbm{k}}}(U, H_S^h)
\end{equation}
coming from representability of the functor $\underline{H_{S}^h}$ (Theorem \ref{thm:Hilbert_scheme}).

Associated to the inclusion $U \lto \underline{H_{S}^h}$ is an element $I \in H_S^h(R)$. This is a homogeneous ideal of $R \otimes S$ with Hilbert function $h$. This in turn corresponds to a closed immersion
\begin{equation}
\bb{U}_{U }= \operatorname{Proj}((R \otimes S)/I) \lto \operatorname{Proj}(R \otimes S) \cong \operatorname{Spec}R \times \bb{S}(A).
\end{equation}
By gluing along all open affines $U \subseteq H_{S}^h$ we obtain a closed subscheme
\begin{equation}
\label{eq:dereliction_embedding}
\iota: \bb{U}_h \lto H_S^h \times \bb{S}(A).
\end{equation}
We post-compose with the product of the Grothendieck embedding $H_S^h \lto \bb{P}^{s_h}$ and the identity $\operatorname{id}: \bb{S}(A) \lto \bb{S}(A)$:
\begin{equation}
\label{eq:universal_embedding}
\bb{U}_h \lto \bb{P}^{s_h} \times \bb{S}(A).
\end{equation}
We post-compose with the canonical inclusion:
\begin{equation}
\bb{U}_h \lto \coprod_{h \in \mathcal{H}}\big(\bb{P}^{s_h} \times \bb{S}(A)\big).
\end{equation}
We post-compose with $\phi^{-1}_{\operatorname{D}}$:
\begin{equation}
\bb{U}_h \lto \bb{S}(?A) \times \bb{S}(A).
\end{equation}
We take $\bb{X}(l)$ to be $\coprod_{h \in \mathcal{H}}\bb{U}_h$. By the universal property of the coproduct, this induces a morphism which we take to be $\iota_l$:
\begin{equation}
\iota_l: \bb{X}(l) = \coprod_{h \in \mathcal{H}}\bb{U}_h \lto \bb{S}(?A) \times \bb{S}(A).
\end{equation}

\textbf{Promotion-link}.
\begin{equation}\label{eq:promotion_link}
\begin{tikzcd}[column sep = small]
	\bullet &&&&& \bullet \\
	& \vdots && \vdots & \vdots \\
	\bullet & \pax & \ldots & \pax & {\prom} & \bullet \\
	& \vdots && \vdots & \vdots
	\arrow["{!B}"', from=3-5, to=4-5]
	\arrow[no head, from=3-5, to=3-6]
	\arrow[no head, from=3-6, to=1-6]
	\arrow[no head, from=1-6, to=1-1]
	\arrow[no head, from=1-1, to=3-1]
	\arrow[no head, from=3-1, to=3-2]
	\arrow[no head, from=3-2, to=3-3]
	\arrow[no head, from=3-3, to=3-4]
	\arrow[no head, from=3-4, to=3-5]
	\arrow["{?A_1}"', from=3-2, to=4-2]
	\arrow["{?A_n}"', from=3-4, to=4-4]
	\arrow["{?A_1}"', from=2-2, to=3-2]
	\arrow["{?A_n}"', from=2-4, to=3-4]
	\arrow["B"', from=2-5, to=3-5]
\end{tikzcd}
\end{equation}

Let $\zeta$ denote the proof net in the interior of the box. That is, let $\zeta$ be the proof net given by the interior of the box and replacing all Pax-links and the Promotion-link by Conclusion-links. Let $\mathcal{L}_\zeta$ denote the set of links of $\zeta$. For each link $l \in \mathcal{L}_\zeta$, the scheme $\bb{X}(l)$ is a subscheme of some product of schemes associated to some edges of $\zeta$. Let $E_l^c$ denote the edges of $\zeta$ which are \textit{not} in $\mathcal{L}_l$, and let $A_e$ denote the formula labelling an edge $e$. Then there is a closed subscheme:
    \begin{equation}
        \prod_{e \in E_l^c}\bb{S}(A_e) \times \bb{X}(l) \lto \bb{S}(\zeta).
    \end{equation}
    We identify $\bb{X}(l)$ with this subscheme. The intersection of the subschemes associated to link $l \in \mathcal{L}_\zeta$ gives a subscheme $\bb{X}(\zeta) = \bigcap_{l \in \mathcal{L}_\zeta}\bb{X}(l) \lto \bb{S}(\zeta)$.

For each $i = 1, \ldots, n$ let $\{s_{h_i}\}_{h_i \in \mathcal{H}}$ denote the set of integers so that
$\bb{S}(?A_i) = \coprod_{h_i \in \mathcal{H}} \bb{P}^{s_{h_i}}$. We fix an element $\textbf{h} = (h_1, \ldots, h_{n}) \in \mathcal{H}^n$. For each $i = 1, \ldots, n$ we let $U_{i} = \operatorname{Spec}R_{i}$ be an open affine chart of $H_{S_i}^{h_i}$, where if $m_i$ denotes the number of unoriented atoms of (necessarily linear) $A_i$, then $S_i = \mathbbm{k}[x_0, \ldots, x_{2^{m_i}-1}]$. Post-compose this with the inclusions $H_{S_i}^{h_i} \lto \bb{P}^{h_{s_i}}$, take the product with $\operatorname{id}: \bb{S}(B) \lto \bb{S}(B)$ and take the product over all $i = 1, \ldots, n$ to obtain
\begin{equation}
\prod_{i = 1}^n \big(U_i \times \bb{S}(B)\big) \lto \prod_{i = 1}^n\big(\bb{P}^{h_{s_i}} \times \bb{S}(B)\big).
\end{equation}
We post-compose this with the canonical inclusion morphisms of the coproduct to obtain
\begin{equation}
\prod_{i = 1}^{n}\big(U_{i} \times \bb{S}(B)\big) \lto \coprod_{\mathbf{h} \in \mathcal{H}^n} \prod_{i = 1}^{n} \big(\bb{P}^{s_{h_i}} \times \bb{S}(B)\big)
\end{equation}
which we post-compose with $\phi^{-1}_{\operatorname{P}^1}$ to obtain
\begin{equation}
\prod_{i = 1}^n \big(U_i \times \bb{S}(B)\big) \lto \prod_{i = 1}^n \big(\bb{S}(?A_i) \times \bb{S}(B)\big).
\end{equation}
We prove in Lemma \ref{lem:well_defined} below that composing $\iota_\zeta: \bb{X}(\zeta) \lto \bb{S}(\zeta)$ with the projection $\rho_{\operatorname{Conc}}: \bb{S}(\zeta) \lto \prod_{i = 1}^n \bb{S}(?A_i) \times \bb{S}(B)$ is a closed immersion
\begin{equation}
\begin{tikzcd}
\bb{X}(\zeta)\arrow[rr, "{\rho_{\operatorname{Conc}}\iota_\zeta}"] & & \prod_{i = 1}^n\big( \bb{S}(?A_i) \times \bb{S}(B)\big).
\end{tikzcd}
\end{equation}
We next consider the scheme $\bb{Y}_{\textbf{h}}$ such that the following is a pullback diagram:
\begin{equation}
\begin{tikzcd}
\prod_{i = 1}^n \big(U_{i} \times \bb{S}(B)\big)\arrow[r] & \prod_{i = 1}^n \big(\bb{S}(?A_i) \times \bb{S}(B)\big)\\
\bb{Y}_{\textbf{h}}\arrow[u]\arrow[r]& \bb{X}(\zeta)\arrow[u, swap, "{\rho_{\operatorname{Conc}}\iota_\zeta}"]
\end{tikzcd}
\end{equation}
Let $R = \bigotimes_{i = 1}^n R_i$ and fix a choice of isomorphism
\begin{equation}
\delta: \prod_{i = 1}^n U_i \lto \operatorname{Spec}R.
\end{equation}
Let $m$ denote the number of unoriented atoms of $B$ and let $S$ denote the graded $\mathbbm{k}$-module $\mathbbm{k}[x_0, \ldots, x_{2^m - 1}]$. Consider the closed immersion
\[\begin{tikzcd}
	{\bb{Y}_{\mathbf{h}}} & {\prod_{i = 1}^{n} \big(U_{i} \times \bb{S}(B)\big)} && {\operatorname{Spec}R \times \bb{S}(B) = \operatorname{Proj}(R \otimes S)}.
	\arrow[from=1-1, to=1-2]
	\arrow["{\delta \times \operatorname{id}_{\bb{S}(B)}}", from=1-2, to=1-4]
\end{tikzcd}\]
There exists a homogeneous saturated ideal $I \subseteq R \otimes S$ such that $\operatorname{Proj}((R \otimes S)/I) \cong \bb{Y}_{\textbf{h}}$. It follows from the proof of Lemma \ref{lem:well_defined} below that for all $d \geq 0$ the $R$-module $(R \otimes S/I)_d$ is locally free of rank $h(d)$, where $h$ is the Hilbert function of $I \subseteq R \otimes S$. Thus $I \in \underline{H_{S}^h}(R)$. Since the Hilbert functor is represented by the scheme $H_S^h$, the ideal $I$ corresponds to a morphism
\begin{equation}
\operatorname{Spec}R \lto H_S^h.
\end{equation}
We pre-compose this with $\delta$ to obtain
\begin{equation}
\label{eq:into_hilbert}
\prod_{i =1}^{n} U_i \lto H_{S}^h.
\end{equation}
This is a morphism depending on choices of open affines $U_{1}, \ldots, U_{n}$ of $H_{S_1}^{h_1}, \ldots, H_{S_n}^{h_n}$ respectively. By ranging over all such choices we obtain a family of morphisms which we can glue to obtain the following:
\begin{equation}
\label{eq:promotion_function}
f: \prod_{i = 1}^{n}H_{S_i}^{h_i} \lto H_{S}^h.
\end{equation}
We consider the graph of this:
\begin{equation}
\Gamma_f \lto \prod_{i = 1}^n \big(H_{S_i}^{h_i} \times H_S^h\big).
\end{equation}
For each $i = 1, \ldots, n$ there is the Grothendieck immersion $H_{S_i}^{h_i}\lto \bb{P}^{s_{h_i}}$. Similarly for each $h \in \mathcal{H}$ there is the Grothendieck immersion $H_S^h \lto \bb{P}^{s_h}$ for some $s_h$. We post-compose with the product of these to obtain
\begin{equation}
\Gamma_f \lto \prod_{i = 1}^n \big(\bb{P}^{s_{h_i}} \times \bb{P}^{s_h}\big).
\end{equation}
We then post-compose with the canonical inclusion morphisms to obtain
\begin{equation}\label{eq:graph_nearly}
\Gamma_{f} \lto \coprod_{\bold{h} \in \mathcal{H}^n}\coprod_{h \in \mathcal{H}} \prod_{i = 1}^n \big(\bb{P}^{h_i} \times \bb{P}^{s_h}\big).
\end{equation}
We take $\bb{X}(l)$ to be $\Gamma_f$. We post-compose with $\phi_{\operatorname{P}^2}^{-1}$ to obtain a morphism which we take to be $\iota_l$:
\begin{equation}
\iota_l: \bb{X}(l) = \Gamma_f \lto \prod_{i = 1}^n\big( \bb{S}(?A_i) \times \bb{S}(!B) \big).
\end{equation}

\textbf{Weakening-link.}
\[\begin{tikzcd}
	\weak \\
	\vdots
	\arrow["{?A}", from=1-1, to=2-1]
\end{tikzcd}\]
We take the empty subscheme $\varnothing$ for $\bb{X}(l)$ and the unique morphism $\varnothing \lto \bb{S}(!A)$ for $\iota_l$:
\begin{equation}
\iota_l: \bb{X}(l) = \varnothing \lto \bb{S}(?A).
\end{equation}

\textbf{Contraction-link.}
\[\begin{tikzcd}[column sep = small]
	\vdots && \vdots \\
	& {\ctr} \\
	& \vdots
	\arrow["{?A}"', curve={height=12pt}, from=1-1, to=2-2]
	\arrow["{?A}", curve={height=-12pt}, from=1-3, to=2-2]
	\arrow["{?A}"', from=2-2, to=3-2]
\end{tikzcd}\]
Let $\Gamma_{\Delta}$ be the graph of the diagonal $\Delta: \bb{S}(?A) \lto \bb{S}(?A) \times \bb{S}(?A)$. We take this to be $\bb{X}(l)$, and $\iota_l$ to be the canonical inclusion:
\begin{equation}
\iota_l: \bb{X}(l) = \Gamma_\Delta \lto \bb{S}(?A) \times \bb{S}(?A) \times \bb{S}(?A).
\end{equation}
\item \textbf{Pax-link}.
\[\begin{tikzcd}
	\vdots \\
	\pax \\
	\vdots
	\arrow["?A"', from=1-1, to=2-1]
	\arrow["?A"', from=2-1, to=3-1]
\end{tikzcd}\]
We define $\bb{X}(l)$ to be the diagonal $\Delta_{\bb{S}(?A)}$ and define $\iota_l$ to be the canonical morphism:
\begin{equation}
\iota_l: \bb{X}(l) = \Delta_{\bb{S}(?A)} \lto \bb{S}(?A) \times \bb{S}(?A).
\end{equation}
\end{defn}

\begin{defn}\label{defn:xpi}
    Let $\pi$ be a proof net with set of links $\mathcal{L}_\pi$. For each link $l \in \mathcal{L}_\pi$, the scheme $\bb{X}(l)$ is a subscheme of some product of schemes associated to some edges of $\pi$ associated with $l$. Let $E_l^c$ denote the edges of $\pi$ which are \textit{not} incident to $l$, and let $A_e$ denote the formula labelling an edge $e$. Then there is a closed subscheme
    \begin{equation}
        \prod_{e \in E_l^c}\bb{S}(A_e) \times \bb{X}(l) \lto \bb{S}(\pi).
    \end{equation}
    We identify $\bb{X}(l)$ with this subscheme. The \textbf{scheme associated to $\pi$ }is the intersection of all schemes associated to the links:
    \begin{equation}
        \bb{X}(\pi) = \bigcap_{l \in \mathcal{L}_\pi}\bb{X}(l).
    \end{equation}
\end{defn}

To prove the prerequisite lemmas used in Definition \ref{def:proof_interpretation} we will make use of the Cartesian product, which is the algebraic equivalent to the product in the category of projective schemes.

\begin{defn}
\label{def:cartesian_product}
Let $S,T$ be graded $\mathbbm{k}$-algebras. We define their \textbf{Cartesian product}, denoted $S \times_{\mathbbm{k}} T$, to be the following graded $\mathbbm{k}$-algebra: as a $\mathbbm{k}$-module it is the sum of the images of the $\mathbbm{k}$-module morphisms $S_d \times_{\mathbbm{k}} T_d \lto S \otimes_{\mathbbm{k}}T$ for all $d \geq 0$.
This is a $\mathbbm{k}$-subalgebra of $S \otimes_{\mathbbm{k}} T$ which is a graded $\mathbbm{k}$-algebra with grading $(S \times_{\mathbbm{k}} T)_d \cong S_d \otimes_{\mathbbm{k}} T_d$ for $d \geq 0$.
\end{defn}

\begin{proposition}
\label{prop:cartesian_commutes_product}
Let $S,T$ be graded $\mathbbm{k}$-algebras, and suppose that $S$ is generated by $S_1$ as an $S_0$-algebra and that $T$ is generated by $T_1$ as a $T_0$-algebra. Then $\operatorname{Proj}(S \times_{\mathbbm{k}} T) \cong \operatorname{Proj}S \times \operatorname{Proj}T$.
\end{proposition}
\begin{proof}
See \cite[Exercise 5.11]{Hartshorne}.
\end{proof}

\begin{lemma}
\label{lem:MLL_projection}
Let $\zeta$ be an MLL proof net with conclusions $A_1, \ldots, A_n$. If $\rho_{\text{Conc}}: \bb{S}(\zeta) \lto \prod_{i = 1}^n \bb{S}(A_i)$ denotes the standard projection, then the composite
\begin{equation}
\begin{tikzcd}
\bb{X}(\zeta)\arrow[r,"{\iota_{\zeta}}"] & \bb{S}(\zeta)\arrow[r,"{\rho_{\text{Conc}}}"] & \prod_{i = 1}^n \bb{S}(A_i)
\end{tikzcd}
\end{equation}
is a closed immersion.

Moreover, if $m_i$ is the number of unoriented atoms of $A_i$, $S$ denotes the graded $\mathbbm{k}$-algebra
\begin{equation}
    S = \mathbbm{k}[x_0,\ldots, x_{2^{m_1}-1}] \times_{\mathbbm{k}} \ldots \times_{\mathbbm{k}} \mathbbm{k}[x_0,\ldots, x_{2^{m_n}-1}]
\end{equation}
and $I \subseteq S$ is the unique saturated homogeneous ideal such that
\begin{equation}
    \operatorname{Proj}(S/I) \cong \bb{X}(\zeta)
\end{equation}
then for all $d \geq 0$, $(S/I)_d$ is locally free.
\end{lemma}
\begin{proof}
We prove the $n = 1$ case where $\zeta$ has a unique conclusion $A = A_1$.

Since $\zeta$ is a proof net, $m = 2m'$ is even. The proof net $\zeta$ admits $m'$ persistent paths which begin with a positively oriented formula. We prove the $m' = 1$ case.

Let $Z_1, \ldots, Z_s$ denote the persistent path which begins with a positively oriented formula. For each formula $B_j$ in $\pi$ there is an integer $t_j \geq 1$ and a sequence of atoms $\textbf{X}_j = X_j^1, \ldots, X_j^{t_j}$ such that the concatenation $\textbf{X}_1, \ldots, \textbf{X}_l$ of these sequences is the persistent path $Z_1, \ldots, Z_s$, with $s = \sum_{j = 1}^l t_j$.

We have commutativity of the following diagram where the diagonal arrow is a closed immersion:
\[\begin{tikzcd}
	{\prod_{j = 1}^l\prod_{i =1 }^{t_l}\bb{P}^1_i} & {\prod_{j=1}^l\bb{S}(B_j) = \bb{S}(\pi)} \\
	{\prod_{i = 1}^{s-1}\Delta_{i,i+1} = \bb{X}(\pi)}
	\arrow[from=1-1, to=1-2]
	\arrow[from=2-1, to=1-1]
	\arrow[from=2-1, to=1-2]
\end{tikzcd}\]
where $\bb{P}^1_i$ denotes the copy of $\bb{P}^1$ pertaining to $Z_i$.

Now consider the following commutative diagram, where the left and middle vertical arrows are induced by the right most vertical arrow $\rho_{\text{conc}}$:
\begin{equation}
\label{eq:leftmost_iso}
\begin{tikzcd}
	{\prod_{i = 1}^{s-1}\Delta_{i,i+1} = \bb{X}(\zeta)} & {\prod_{j = 1}^l\prod_{i =1 }^{t_l}\bb{P}^1} & {\prod_{j=1}^l\bb{S}(B_j) = \bb{S}(\zeta)} \\
	{\Delta_{1,s}} & {\bb{P}^1_1 \times \bb{P}^1_s} & {\bb{P}^3 = \bb{S}(A)}
	\arrow[from=1-1, to=1-2]
	\arrow[curve={height=-24pt}, from=1-1, to=1-3]
	\arrow[from=1-1, to=2-1]
	\arrow[from=1-2, to=1-3]
	\arrow[from=1-2, to=2-2]
	\arrow["{\rho_{\text{conc}}}", from=1-3, to=2-3]
	\arrow[from=2-1, to=2-2]
	\arrow["{\operatorname{Seg}}", from=2-2, to=2-3]
\end{tikzcd}
\end{equation}
The first claim follows from the observation that the left most vertical morphism in \ref{eq:leftmost_iso} is an isomorphism, and the bottom horizontal morphisms are closed immersions.

We have computed $\bb{X}(\zeta)$ as $\Delta_{1,s} \lto \bb{P}^3$, so we can write down generators for $I$ explicitly:
\begin{equation}
I = (Z_{00}Z_{11} - Z_{01}Z_{10}, Z_{01} - Z_{10})\subseteq \mathbbm{k}[Z_{00}, Z_{01}, Z_{10}, Z_{11}].
\end{equation}
For each $d \geq 0$ the module $(\mathbbm{k}[Z_{00}, Z_{01}, Z_{10}, Z_{11}]/I)_d$ is a free $\mathbbm{k}$-module.
\end{proof}

The following lemma is used in Definition \ref{def:MELL_interpretation}, in the Promotion-link clause. There, we used the fact that the interior of a box inside a shallow proof net corresponds to an ideal of the right from to induce a morphism into the projective scheme representing the Hilbert scheme. Since we only consider shallow proofs, we do not allow for Weakening-links to appear inside boxes (as nested boxes are forbidden), so we may assume in the following lemma that $\zeta$ admits no Weakening-links.

\begin{lemma}
\label{lem:well_defined}
Let $\zeta$ be a nearly linear proof with conclusions $?A_1, \ldots, ?A_n, B$. For each $i = 1, \ldots, n$ let $\{s_{h_i}\}_{h_i \in \mathcal{H}}$ denote the set of integers so that
$\bb{S}(?A_i) = \coprod_{h_i \in \mathcal{H}} \bb{P}^{s_{h_i}}$. Let $\mathbf{h} = (h_1, \ldots, h_n)$ be an element of $\mathcal{H}^n$. Let $\rho_{\textbf{Conc}}: \bb{S}(\zeta) \lto \prod_{i = 1}^n \bb{S}(?A_i) \times \bb{S}(B)$ denote the standard projection. Consider the following pullback diagram:
\[\begin{tikzcd}
	{\bb{X}(\zeta)} & {\bb{S}(\zeta)} & {\prod_{i = 1}^n\bb{S}(?A_i) \times \bb{S}(B)} \\
	&& {\coprod_{\bold{h} \in \mathcal{H}^n}\prod_{i = 1}^n(\bb{P}^{s_{h_i}} \times \bb{S}(B))} \\
	{\bb{X}(\zeta) \cap \prod_{i = 1}^n(\bb{P}^{s_{h_i}} \times \bb{S}(B))} && {\prod_{i = 1}^n(\bb{P}^{s_{h_i}} \times \bb{S}(B))}
	\arrow["{\iota_{\zeta}}", from=1-1, to=1-2]
	\arrow["{\rho_{\conc}}", from=1-2, to=1-3]
	\arrow["{\phi_{\operatorname{P}^1}^{-1}}"', from=2-3, to=1-3]
	\arrow[from=3-1, to=1-1]
	\arrow[from=3-1, to=3-3]
	\arrow["{\iota_{\bold{h}}}"', from=3-3, to=2-3]
\end{tikzcd}\]
where $\iota_{\textbf{h}}$ is the standard inclusion and $\phi^{-1}_{\operatorname{P}^1}$ is the isomorphism of \eqref{eq:first_promotion_iso}. Then the composite
\[\begin{tikzcd}[column sep = small]
	{\bb{X}(\zeta) \cap \prod_{i = 1}^n(\bb{P}^{s_{h_i}} \times \bb{S}(B))} & {\prod_{i = 1}^n(\bb{P}^{s_{h_i}} \times \bb{S}(B))} & {\coprod_{\mathbf{h} \in {\mathcal{H}}^n}\prod_{i = 1}^n(\bb{P}^{s_{h_i}} \times \bb{S}(B))} & {\prod_{i = 1}^n\bb{S}(?A_i) \times \bb{S}(B)}
	\arrow[from=1-1, to=1-2]
	\arrow["{\iota_{\bold{h}}}"', from=1-2, to=1-3]
	\arrow["{\phi_{\operatorname{P}^1}^{-1}}"', from=1-3, to=1-4]
\end{tikzcd}\]
is a closed immersion. Moreover, if $m$ denotes the number of unoriented atoms in $B$ and $S$ denotes the graded $\mathbbm{k}$-algebra
\begin{equation}
    S = \mathbbm{k}[x_0, \ldots, x_{s_{h_1}}] \times_{\mathbbm{k}} \ldots \times_{\mathbbm{k}} \mathbbm{k}[x_0, \ldots, x_{s_{h_n}}] \times_{\mathbbm{k}} \mathbbm{k}[x_0, \ldots, x_{m}]
\end{equation}
and $I$ the unique saturated homogeneous ideal such that
\begin{equation}
    \operatorname{Proj}(S/I) \cong \bb{X}(\zeta)
\end{equation}
\end{lemma}
then for all $d \geq 0$, $(S/I)_d$ is locally free.
\begin{proof}
Contraction-links trivially introduce isomorphisms which may be ignored, and so we assume that $\zeta$ is Contraction-free. Thus, $\zeta$ is of the following form:
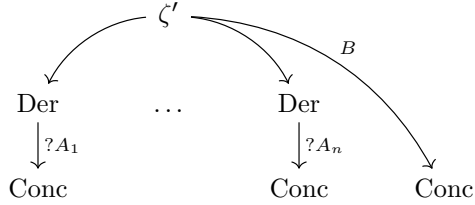
\[\begin{tikzcd}
	& {\zeta'} \\
	\der & \ldots & \der \\
	\conc && \conc & \conc
	\arrow[curve={height=12pt}, from=1-2, to=2-1]
	\arrow[curve={height=-12pt}, from=1-2, to=2-3]
	\arrow["B", curve={height=-24pt}, from=1-2, to=3-4]
	\arrow["{?A_1}", from=2-1, to=3-1]
	\arrow["{?A_n}", from=2-3, to=3-3]
\end{tikzcd}\]
where $\zeta'$ is the linear part (Definition \ref{def:nearly_linear}) of the nearly linear proof $\zeta$.

By Lemma \ref{lem:MLL_projection} we have that the canonical projection
\begin{equation}
    \bb{X}(\zeta') \lto \prod_{i = 1}^n\bb{S}(A_i) \times \bb{S}(B)
\end{equation}
is a closed immersion. For each $i = 1, \ldots, n$ let $\bb{U}_{h_i} \lto \bb{P}^{s_{h_i}} \times \bb{S}(A_i)$ 
be the closed immersion of the universal closed subscheme as given in \eqref{eq:universal_embedding}. Then
\[\begin{tikzcd}
	{\prod_{i = 1}^n\bb{U}_{h_i} \cap \bb{X}(\zeta') \cap \prod_{i = 1}^n(\bb{P}^{s_{h_i}} \times \bb{S}(B))} & {\prod_{i = 1}^n(\bb{P}^{s_{h_i}} \times \bb{S}(A_i) \times \bb{S}(B))}
	\arrow[from=1-1, to=1-2]
\end{tikzcd}\]
is a closed immersion (being the pullback of two closed immersions).

Each $A_i$ as well as $B$ is linear. Let $m_1, \ldots, m_n,m$ respectively denote the number of unoriented atoms of $A_1, \ldots, A_n, B$. Let $f: \bb{N} \lto \bb{N}$ denote the function given by the equation $f(n) \longmapsto 2^n-1$, then
\begin{equation}
    \bb{S}(A_i) = \bb{P}^{f(m_i)},\quad \bb{S}(B) = \bb{P}^{f(m)}.
\end{equation}
The integer $\sum_{i=1}^nm_i + m$ is the number of unoriented atoms in the conclusions of a MLL proof net, and thus is necessarily even. Let $p$ denote half this number. We saw in the proof of Lemma \ref{lem:MLL_projection} that there is a factorisation:
\[\begin{tikzcd}
	{\bb{X}(\zeta') \cap \prod_{i = 1}^n(\bb{S}(A_i) \times \bb{S}(B))} & {\prod_{i = 1}^n(\bb{S}(A_i) \times \bb{S}(B))} \\
	& {(\bb{P}^1)^p}
	\arrow[from=1-1, to=1-2]
	\arrow[from=1-1, to=2-2]
	\arrow[from=2-2, to=1-2]
\end{tikzcd}\]
This implies that there exists the following factorisation:
\[\begin{tikzcd}
	{\prod_{i = 1}^n \bb{U}_{h_i} \cap \bb{X}(\zeta') \cap \prod_{i = 1}^n(\bb{P}^{s_{h_i}} \times \bb{S}(B))} & {\prod_{i = 1}^n(\bb{P}^{s_{h_i}} \times \bb{S}(A_i) \times \bb{S}(B))} \\
	& {\big(\prod_{i = 1}^n\bb{U}_{h_i} \times \bb{S}(B)\big) \cap \big(\prod_{i = 1}^n \bb{P}^{s_{h_j}} \times (\bb{P}^1)^p\big)}
	\arrow[from=1-1, to=1-2]
	\arrow[from=1-1, to=2-2]
	\arrow[from=2-2, to=1-2]
\end{tikzcd}\]
We have assumed that $\zeta$ is shallow, and so all persistent paths of $\zeta'$ go through $B$. This implies that $(\bb{P}^1)^p \lto \prod_{i = 1}^n\bb{S}(A_i) \times \bb{S}(B)$ can be realised as a closed subscheme of $\bb{S}(B)$. We thus have the following factorisation:
\[\begin{tikzcd}
	{(\bb{P}^1)^p} & {\prod_{i = 1}^n\bb{S}(A_i) \times \bb{S}(B)} \\
	& {\bb{S}(B)}
	\arrow[from=1-1, to=1-2]
	\arrow[from=1-1, to=2-2]
	\arrow["{\operatorname{Projection}}", from=1-2, to=2-2]
\end{tikzcd}\]
where the diagonal arrow is a closed immersion. This implies the existence of the following factorisation:
\[\begin{tikzcd}
	{\prod_{i = 1}^n(\bb{P}^{s_{h_i}} \times \bb{S}(A_i) \times \bb{S}(B))} & {\prod_{i = 1}^n(\bb{P}^{s_{h_i}}\times \bb{S}(B))} \\
	{\prod_{i = 1}^n \bb{U}_{h_i} \cap (\bb{P}^1)^p}
	\arrow["\rho", from=1-1, to=1-2]
	\arrow[from=2-1, to=1-1]
	\arrow["\alpha"', from=2-1, to=1-2]
\end{tikzcd}\]
where $\rho$ is a product of projections and $\alpha$ is a closed immersion.

The final claim follows from the fact that $\prod_{i = 1}^n \bb{U}_{h_i} \cap (\bb{P}^1)^p \lto \prod_{i = 1}^n\bb{S}(?A_i) \times \bb{S}(B)$ is an intersection of closed subschemes of the right form.
\end{proof}

\begin{remark}
\label{rmk:extension}
We remark that Lemma \ref{lem:well_defined} is the main hurdle in extending our model beyond shallow proofs and to all MELL proofs. We commented on this already in Remark \ref{rmk:equations} and we will again in Section \ref{sec:future}.
\end{remark}

\subsection{Invariance under cut-elimination}
Now that we have defined the locally closed subscheme $\bb{X}(\pi)$ of $\bb{S}(\pi)$ to each shallow proof net $\pi$, we now move onto considering how the schemes $\bb{X}(\pi)$ and $\bb{X}(\pi')$ are related if $\pi'$ is a shallow proof net obtained from $\pi$ via single step cut-reduction. The main result of this paper is Theorem \ref{thm:main} which extends \cite[Proposition 4.6]{AlgPnt}.

\begin{defn}
\label{def:S_T}
For each reduction $\gamma: \pi \lto \pi'$ we define a closed subscheme $\bb{Y}(\pi') \subseteq \bb{S}(\pi')$ and a pair of morphisms of schemes $S_\gamma: \bb{S}(\pi) \lto \bb{S}(\pi')$, $T_\gamma: \bb{Y}(\pi') \lto \bb{S}(\pi)$.

Let $\gamma: \pi \lto \pi'$ be a reduction. Let $\mathcal{E}_\pi$ denote the set of edges of $\pi$, and $\mathcal{E}_{\pi'}$ that of $\pi'$.

\textbf{$\gamma: \pi \lto \pi'$ is an $\ax/\cut$-reduction.} We set $\bb{Y}(\pi') = \bb{S}(\pi')$. Consider the following reduction where the labels $a,b,c,d$ are artificial:
\begin{equation}\label{eq:ax/cut_red}
\begin{tikzcd}[column sep = small]
	& \ax && \vdots \\
	\vdots && \cut
	\arrow["{A_b}", from=1-2, to=2-3]
	\arrow["{\neg A_c}", curve={height=-12pt}, from=1-4, to=2-3]
	\arrow["{\neg A_a}"', curve={height=12pt}, from=1-2, to=2-1]
\end{tikzcd}
\quad
\stackrel{\gamma}{\lto}
\quad
%
\begin{tikzcd}[column sep = small]
	\vdots \\
	\vdots
	\arrow["{\neg A_d}"', from=1-1, to=2-1]
\end{tikzcd}
\end{equation}
For the edge $e$ of $\pi'$ labelled $\neg A_d$ in \eqref{eq:ax/cut_red}, let $\rho_e$ denote the projection $\rho_e: \bb{S}(\pi) \lto \bb{S}(\neg A_c)$. For every edge $e$ of $\pi'$ which is not displayed in \eqref{eq:ax/cut_red} there is a corresponding edge $e'$ of $\pi$. For these, set $\rho_e$ to be the projection $\rho_e: \bb{S}(\pi) \lto \bb{S}(A_{e'})$. We define $S_{\gamma}: \bb{S}(\pi) \lto \bb{S}(\pi')$ to be the morphism induced by the universal property of the product and the set $\{\rho_e\}_{e \in \mathcal{E}_{\pi'}}$.

For the edges $e$ of $\pi$ displayed in \eqref{eq:ax/cut_red} labelled $\neg A_a, A_b, \neg A_c$, let $\tau_e$ denote the projection $\bb{S}(\pi') \lto \bb{S}(\neg A_d)$. For every edge $e$ of $\pi$ which is not displayed in \eqref{eq:tensor_red} there is a corresponding edge $e'$ of $\pi'$. For these set $\tau_e$ to be the projection $\tau_e: \bb{S}(\pi') \lto \bb{S}(A_{e'})$. We define $T_{\gamma}: \bb{S}(\pi') \lto \bb{S}(\pi)$ to be the morphism induced by the universal property of the product and the set $\{\tau_e\}_{e \in \mathcal{E}_{\pi}}$.

\textbf{$\gamma: \pi \lto \pi'$ is a $\otimes/\parr$-reduction.} We set $\bb{Y}(\pi') = \bb{S}(\pi')$. Consider the following reduction where the labels are artificial:
\begin{equation}
\label{eq:tensor_red}
\begin{tikzcd}[column sep = small]
	\vdots && \vdots && \vdots && \vdots \\
	& \otimes &&&& \parr \\
	&&& \cut \\
	&&& {\stackrel{\gamma}{\lto}} \\
	\vdots && \vdots && \vdots && \vdots \\
	&&& \cut \\
	&&& \cut
	\arrow["{A_a}"', curve={height=12pt}, from=1-1, to=2-2]
	\arrow["{B_b}", curve={height=-12pt}, from=1-3, to=2-2]
	\arrow["{\neg B_c}"', curve={height=12pt}, from=1-5, to=2-6]
	\arrow["{\neg A_d}", curve={height=-12pt}, from=1-7, to=2-6]
	\arrow["{(A \otimes B)_f}"', curve={height=12pt}, from=2-2, to=3-4]
	\arrow["{(\neg B \parr \neg A)_g}", curve={height=-12pt}, from=2-6, to=3-4]
	\arrow["{A_j}"', curve={height=18pt}, from=5-1, to=7-4]
	\arrow["{B_h}"', curve={height=12pt}, from=5-3, to=6-4]
	\arrow["{\neg B_i}", curve={height=-12pt}, from=5-5, to=6-4]
	\arrow["{\neg A_k}", curve={height=-18pt}, from=5-7, to=7-4]
\end{tikzcd}
\end{equation}
For the edges of $\pi'$ displayed in \eqref{eq:tensor_red} define a morphism $\rho_e$ to be a projection according to the following table:
\begin{center}
\begin{tabular}{| c | c |}
\hline
\textbf{Edge label} & $\rho_e$\\
\hline
$B_h$ & $\bb{S}(\pi) \lto \bb{S}(B_b)$\\
\hline
$\neg B_i$ & $\bb{S}(\pi) \lto \bb{S}(\neg B_c)$\\
\hline
$A_j$ & $\bb{S}(\pi) \lto \bb{S}(A_a)$\\
\hline
$\neg A_k$ & $\bb{S}(\pi) \lto \bb{S}(\neg A_d)$\\
\hline
\end{tabular}
\end{center}
For every edge $e$ of $\pi'$ which is not displayed in \eqref{eq:tensor_red} there is a corresponding edge $e'$ of $\pi$. For these set $\rho_e$ to be the projection $\bb{S}(\pi) \lto \bb{S}(A_{e'})$. We define $S_{\gamma}: \bb{S}(\pi) \lto \bb{S}(\pi')$ to be the morphism induced by the universal property of the product and the set $\{\rho_e\}_{e \in \mathcal{E}_{\pi'}}$.

For the following edges of $\pi$ displayed in \eqref{eq:tensor_red} define a morphism $\tau_e$ to be a projection according to the following table:

\begin{center}
\begin{tabular}{| c | c |}
\hline
\textbf{Edge label} & $\tau_e$\\
\hline
$A_a$ & $\bb{S}(\pi') \lto \bb{S}(A_j)$\\
\hline
$B_b$ & $\bb{S}(\pi') \lto \bb{S}(B_h)$\\
\hline
$\neg B_c$ & $\bb{S}(\pi') \lto \bb{S}(\neg B_i)$\\
\hline
$\neg A_d$ & $\bb{S}(\pi') \lto \bb{S}(\neg A_k)$\\
\hline
\end{tabular}
\end{center}
For the edge $e$ labelled $(A \otimes B)_f$, say $\bb{S}(A) = \coprod_{i \in I}\bb{P}^{r_i}, \bb{S}(B) = \coprod_{j \in J}\bb{P}^{s_j}$. For each pair $(i,j) \in I \times J$ we consider the Segre embedding
\begin{equation}
\operatorname{Seg}: \bb{P}^{r_i} \times \bb{P}^{s_j} \lto \bb{P}^{(r_i + 1)(s_j + 1)-1}.
\end{equation}
We post-compose this with the canonical inclusion to obtain
\begin{equation}
\bb{P}^{r_i} \times \bb{P}^{s_j} \lto \coprod_{i \in I}\coprod_{j \in J}\bb{P}^{(r_i + 1)(s_j + 1)-1} = \bb{S}(A \otimes B).
\end{equation}
By the universal property of the coproduct we obtain
\begin{equation}
\coprod_{i \in I}\coprod_{j \in J}\bb{P}^{r_i} \times \bb{P}^{s_j} \lto \bb{S}(A \otimes B)
\end{equation}
which we pre-compose with $\phi_{\operatorname{M}}$ of Definition \ref{def:MELL_interpretation} to obtain
\begin{equation}
\bb{S}(A) \times \bb{S}(B) \lto \bb{S}(A \otimes B).
\end{equation}
We set this to be $\tau_e$. We define $\tau_e$ similarly when $e$ is the edge labelleing $(\neg B \parr \neg A)_g$. We define $T_\gamma: \bb{S}(\pi') \lto \bb{S}(\pi)$ to be the morphism induced by the universal property of the product and the set $\{\tau_e\}_{e \in \mathcal{E}_{\pi}}$.

\textbf{$\gamma: \pi \lto \pi'$ is a $!/?$-reduction.} Consider the following reduction:
\begin{equation}\label{eq:?=!-red}
\begin{tikzcd}[column sep = small]
	& \bullet &&& \bullet \\
	\vdots && \vdots & \vdots \\
	{?} & \bullet & {\prom} & \pax & \bullet \\
	& {\cut} && {\operatorname{c}} \\
	\vdots && \vdots & \vdots \\
	& {\cut} && {\operatorname{c}}
	\arrow[no head, from=3-3, to=3-4]
	\arrow[no head, from=3-3, to=3-2]
	\arrow["{?B_c}", from=2-4, to=3-4]
	\arrow["{?B_g}", from=3-4, to=4-4]
	\arrow[no head, from=1-2, to=3-2]
	\arrow[no head, from=1-2, to=1-5]
	\arrow[no head, from=1-5, to=3-5]
	\arrow["{A_b}", from=2-3, to=3-3]
	\arrow["{!A_f}", curve={height=-12pt}, from=3-3, to=4-2]
	\arrow["{?\neg A_d}"', curve={height=12pt}, from=3-1, to=4-2]
	\arrow["{\neg A_a}"', from=2-1, to=3-1]
	\arrow["{?B_j}", from=5-4, to=6-4]
	\arrow["{A_i}", curve={height=-12pt}, from=5-3, to=6-2]
	\arrow["{\neg A_h}"', curve={height=12pt}, from=5-1, to=6-2]
	\arrow[no head, from=3-5, to=3-4]
\end{tikzcd}
\end{equation}
Say $\bb{S}(?B_g) = \coprod_{h \in \mathcal{H}} \bb{P}^{s_h}$. Then there exists a graded $\mathbbm{k}$-algebra $S$, such that for each Hilbert function $h \in \mathcal{H}$, there is a fixed choice of closed immersion $H_S^h \lto \bb{P}^{s_h}$. We set $\bb{Y}(\pi') = \coprod_{h \in \mathcal{H}}H_S^{h_s} \cap \bb{S}(\pi')$.

For the edges $e$ of $\pi'$ displayed in \eqref{eq:tensor_red} define a morphism $\rho_e$ to be a projection according to the following table:
\begin{center}
\begin{tabular}{| c | c |}
\hline
\textbf{Edge label} & $\rho_e$\\
\hline
$\neg A_h$ & $\bb{S}(\pi) \lto \bb{S}(\neg A_a)$\\
\hline
$A_i$ & $\bb{S}(\pi) \lto \bb{S}(A_b)$\\
\hline
$?B_j$ & $\bb{S}(\pi) \lto \bb{S}(?B_c)$\\
\hline
\end{tabular}
\end{center}
For every edge $e$ of $\pi'$ which is not displayed in \eqref{eq:?=!-red} there is a corresponding edge $e'$ of $\pi$. For these we set $\rho_e$ to be the projection $\bb{S}(\pi) \lto \bb{S}(A_{e'})$. We define $S_{\gamma}: \bb{S}(\pi) \lto \bb{S}(\pi')$ to be the morphism given by the universal property of the product and the set $\{\rho_e\}_{e \in \mathcal{E}_{\pi'}}$.

For the following edges $e$ of $\pi$ displayed in \eqref{eq:?=!-red} we define a morphism $\tau_e$ to be a projection according to the following table:
\begin{center}
\begin{tabular}{| c | c |}
\hline
\textbf{Edge label} & $\tau_e$\\
\hline
$\neg A_a$ & $\bb{S}(\pi') \lto \bb{S}(\neg A_h)$\\
\hline
$A_b$ & $\bb{S}(\pi') \lto \bb{S}(A_i)$\\
\hline
$?B_c$ & $\bb{S}(\pi') \lto \bb{S}(?B_j)$\\
\hline
$?B_g$ & $\bb{S}(\pi') \lto \bb{S}(?B_j)$\\
\hline
\end{tabular}
\end{center}
Let $\zeta$ denote the proof net in the interior of the displayed box. We have already seen in Definition \ref{def:MELL_interpretation} that if $\bb{S}(!A) = \coprod_{h \in \mathcal{H}}\bb{P}^{s_h}$ and if we are given an element $h_1 \in \mathcal{H}$, we can construct graded $\mathbbm{k}$-algebras $S_1, S$ along with a morphism
\begin{equation}
f: H_{S_1}^{h_1} \lto H_S^h
\end{equation}
as in \eqref{eq:promotion_function}. We post-compose with the canonical inclusion to obtain
\begin{equation}
H_{S_1}^{h_1} \lto \coprod_{h \in \mathcal{H}} H_S^h \cong \bb{S}(!A).
\end{equation}
By the universal property of the disjoint union we obtain
\begin{equation}
\coprod_{h_1 \in \mathcal{H}}H_{S_1}^{h_1} \lto \bb{S}(!A)
\end{equation}
which we pre-compose with $(\phi_{\operatorname{P}^2}\vert_{\coprod_{h_1 \in \mathcal{H}}H_{S_1}^{h_1}})^{-1}$, which is the inverse of a restriction of $\phi_{\operatorname{P}^2}$ of Definition \ref{def:MELL_interpretation} in order to obtain
\begin{equation}
\coprod_{h_1 \in \mathcal{H}}H_{S_1}^{h_1} \cap (\bb{S}(?B_j) \times \bb{S}(A_i)) \lto \bb{S}(!A).
\end{equation}
We take $\tau_e$, for $e$ the edges labelled $!A_f, ?\neg A_d$, to be the result of pre-composing this with (the restriction of) the projection $\bb{S}(\pi') \lto \bb{S}(?B_j) \times \bb{S}(A_i)$:
\begin{equation}
\tau_e: \bb{Y}(\pi') \lto \bb{S}(!A).
\end{equation}
We define $T_\gamma: \bb{Y}(\pi') \lto \bb{S}(\pi)$ to be the morphism induced by the universal property of the product and the set $\{\tau_e\}_{e \in \mathcal{E}_{\pi}}$.

\textbf{$\gamma: \pi \lto \pi'$ is a $\weak/!$-reduction.}
\begin{equation}
\label{eq:weak/!}
\begin{tikzcd}[column sep = small]
	& \bullet &&& \bullet \\
	&& \vdots & \vdots \\
	\weak & \bullet & {\prom} & \pax & \bullet \\
	& \cut && \vdots \\
	&&& \lto \\
	&&& \weak \\
	&&& \vdots
	\arrow["{?\neg A_a}"', curve={height=12pt}, from=3-1, to=4-2]
	\arrow["{!A_b}", curve={height=-12pt}, from=3-3, to=4-2]
	\arrow[no head, from=3-2, to=3-3]
	\arrow[no head, from=3-3, to=3-4]
	\arrow[no head, from=3-5, to=1-5]
	\arrow[no head, from=1-5, to=1-2]
	\arrow[no head, from=1-2, to=3-2]
	\arrow["{A_c}", from=2-3, to=3-3]
	\arrow["{?B_{d}}", from=2-4, to=3-4]
	\arrow["{?B_f}", from=3-4, to=4-4]
	\arrow["{?B_g}", from=6-4, to=7-4]
	\arrow[no head, from=3-4, to=3-5]
\end{tikzcd}
\end{equation}
Let $\varnothing_g \lto \bb{S}(?B_g)$ denote the empty subscheme. We set $\bb{Y}(\pi') = \varnothing_g \cap \bb{S}(\pi')$.

For the edge $e$ of \eqref{eq:weak/!} labelled $?B_g$, define $\rho_e$ to be the projection $\rho_e: \bb{S}(\pi) \lto \bb{S}(?B_f)$. For every edge $e$ of $\pi'$ which is not displayed in \eqref{eq:?=!-red} there is a corresponding edge $e'$ of $\pi$. For these we set $\rho_e$ to be the projection $\bb{S}(\pi) \lto \bb{S}(A_{e'})$. We define $S_{\gamma}: \bb{S}(\pi) \lto \bb{S}(\pi')$ to be the morphism induced by the universal property of the product and the set $\{\rho_e\}_{e \in \mathcal{E}_{\pi'}}$.

Let $\zeta$ denote the proof inside the box. The empty scheme $\varnothing_g$ is the initial object in the category of schemes over $\mathbbm{k}$. For each edge $e$ of $\zeta$ we define $\tau_{e}$ to be the unique morphism $\tau_e: \varnothing_g \lto \bb{S}(A_{e})$. If $e$ is labelled $?\neg A_a$ or $!A_b$ we similarly define $\tau_e$ to be the unique morphism $\tau_e: \varnothing_g \lto \bb{S}(A_{e})$. For every edge $e$ of $\pi$ which is not displayed in \eqref{eq:weak/!} there is a corresponding edge $e'$ of $\pi'$. For these we set $\tau_e$ to be the projection $\tau_e: \bb{S}(\pi') \lto \bb{S}(A_{e'})$. We define $T_\gamma: \bb{Y}(\pi') \lto \bb{S}(\pi)$ to be the morphism induced by the universal property of the product and the set $\{\tau_e\}_{e \in \mathcal{E}_\pi}$.

\textbf{$\gamma: \pi \lto \pi'$ is a $\ctr/!$-reduction.} Set $\bb{Y}(\pi') = \bb{S}(\pi')$.
\begin{equation}
\label{eq:ctr/!_reduction}
\begin{tikzcd}[column sep = small]
	&&&&& \bullet &&& \bullet \\
	& \vdots && \vdots &&& \vdots & \vdots \\
	&& {\ctr} &&& \bullet & {\prom} & {\pax} & \bullet \\
	&&&& {\cut} &&& \vdots \\
	& \bullet &&& \bullet && \bullet &&& \bullet \\
	&& \vdots & \vdots &&&& \vdots & \vdots \\
	\vdots & \bullet & {\prom} & {\pax} & \bullet & \vdots & \bullet & {\prom} & {\pax} & \bullet \\
	& {\cut} &&&&& {\cut} \\
	&&&&& {\ctr} \\
	&&&&& \vdots
	\arrow["{!A_g}", curve={height=-12pt}, from=3-7, to=4-5]
	\arrow["{A_c}", no head, from=2-7, to=3-7]
	\arrow["{A_i}", no head, from=6-3, to=7-3]
	\arrow["{A_k}", no head, from=6-8, to=7-8]
	\arrow["{!A_n}", curve={height=-12pt}, from=7-3, to=8-2]
	\arrow["{!A_q}", curve={height=-12pt}, from=7-8, to=8-7]
	\arrow[no head, from=3-8, to=3-7]
	\arrow[no head, from=3-7, to=3-6]
	\arrow[no head, from=3-6, to=1-6]
	\arrow[no head, from=1-6, to=1-9]
	\arrow[no head, from=1-9, to=3-9]
	\arrow["{?B_d}", from=2-8, to=3-8]
	\arrow["{?B_h}", from=3-8, to=4-8]
	\arrow["{?B_j}", from=6-4, to=7-4]
	\arrow[no head, from=7-2, to=7-3]
	\arrow[no head, from=7-3, to=7-4]
	\arrow[no head, from=7-5, to=5-5]
	\arrow[no head, from=5-2, to=5-5]
	\arrow[no head, from=5-2, to=7-2]
	\arrow[no head, from=7-7, to=7-8]
	\arrow[no head, from=7-8, to=7-9]
	\arrow[no head, from=7-10, to=5-10]
	\arrow[no head, from=5-10, to=5-7]
	\arrow[no head, from=5-7, to=7-7]
	\arrow["{?B_l}", from=6-9, to=7-9]
	\arrow["{?B_o}", curve={height=12pt}, from=7-4, to=9-6]
	\arrow["{?B_r}", curve={height=-18pt}, from=7-9, to=9-6]
	\arrow["{?B_s}", from=9-6, to=10-6]
	\arrow[no head, from=7-4, to=7-5]
	\arrow[no head, from=7-9, to=7-10]
	\arrow[no head, from=3-8, to=3-9]
	\arrow["{?\neg A_m}"', curve={height=12pt}, from=7-1, to=8-2]
	\arrow["{?\neg A_p}"', curve={height=12pt}, from=7-6, to=8-7]
	\arrow["{?\neg A_f}"', curve={height=12pt}, from=3-3, to=4-5]
	\arrow["{?\neg A_a}"', curve={height=12pt}, from=2-2, to=3-3]
	\arrow["{?\neg A_b}", curve={height=-12pt}, from=2-4, to=3-3]
\end{tikzcd}
\end{equation}
For the edges $e$ of $\pi'$ displayed in \eqref{eq:ctr/!_reduction} define a morphism $\rho_e$ to be a projection according to the following table:
\begin{center}
\begin{tabular}{| c | c |}
\hline
\textbf{Edge label} & $\rho_e$\\
\hline
$?\neg A_m$ & $\bb{S}(\pi) \lto \bb{S}(?A_g)$\\
\hline
$!A_n$ & $\bb{S}(\pi) \lto \bb{S}(!A_g)$\\
\hline
$A_i$ & $\bb{S}(\pi) \lto \bb{S}(A_c)$\\
\hline
$?B_j$ & $\bb{S}(\pi) \lto \bb{S}(?B_d)$\\
\hline
$?B_o$ & $\bb{S}(\pi) \lto \bb{S}(?B_h)$\\
\hline
$?\neg A_p$ & $\bb{S}(\pi) \lto \bb{S}(?\neg A_g)$\\
\hline
$!A_q$ & $\bb{S}(\pi) \lto \bb{S}(!A_g)$\\
\hline
$A_k$ & $\bb{S}(\pi) \lto \bb{S}(A_c)$\\
\hline
$?B_l$ & $\bb{S}(\pi) \lto \bb{S}(?B_d)$\\
\hline
$?B_r$ & $\bb{S}(\pi) \lto \bb{S}(?B_h)$\\
\hline
$?B_s$ & $\bb{S}(\pi) \lto \bb{S}(?B_h)$\\
\hline
\end{tabular}
\end{center}
For every edge $e$ of $\pi'$ which is not displayed in \eqref{eq:ctr/!_reduction} there is a corresponding edge $e'$ of $\pi$. For these we set $\rho_e$ to be the projection $\rho_e: \bb{S}(\pi) \lto \bb{S}(A_{e'})$. We define $S_{\gamma}: \bb{S}(\pi) \lto \bb{S}(\pi')$ to be the morphism induced by the universal property of the product and the set $\{\rho_e\}_{e \in \mathcal{E}_{\pi'}}$.

For the edges $e$ of $\pi'$ displayed in \eqref{eq:ctr/!_reduction} we define a morphism $\tau_e$ to be a projection according to the following table:

\begin{center}
\begin{tabular}{| c | c |}
\hline
\textbf{Edge label} & $\tau_e$\\
\hline
$?A_a$ & $\bb{S}(\pi') \lto \bb{S}(?A_n)$\\
\hline
$?\neg A_b$ & $\bb{S}(\pi') \lto \bb{S}(!A_n)$\\
\hline
$A_c$ & $\bb{S}(\pi') \lto \bb{S}(A_i)$\\
\hline
$?B_d$ & $\bb{S}(\pi') \lto \bb{S}(?B_j)$\\
\hline
$?\neg A_f$ & $\bb{S}(\pi') \lto \bb{S}(?A_n)$\\
\hline
$!A_g$ & $\bb{S}(\pi') \lto \bb{S}(!A_n)$\\
\hline
$?B_h$ & $\bb{S}(\pi') \lto \bb{S}(?B_o)$\\
\hline
\end{tabular}
\end{center}

For every edge $e$ of $\pi$ which is not displayed in \eqref{eq:ctr/!_reduction} there is a corresponding edge $e'$ of $\pi'$. For these we set $\tau_e$ to be the projection $\tau_e: \bb{S}(\pi') \lto \bb{S}(A_{e'})$. We define $T_\gamma: \bb{Y}(\pi') \lto \bb{S}(\pi)$ to be the morphism induced by the universal property of the product and the set $\{\rho_e\}_{e \in \mathcal{E}_\pi}$.
\end{defn}

\begin{remark}
We only needed to introduce the restriction $\bb{S}(\pi')\vert_{\bb{Y}(\pi')}$ in Definition \ref{def:S_T} for $!/?$-reductions and $!/\weak$-reductions. It can be checked easily that given a reduction $\pi \stackrel{\gamma}{\lto} \pi'$ we have
\begin{equation}
    T_{\gamma}^{-1}(\bb{Y}(\pi)) \subseteq \bb{Y}(\pi')
\end{equation}
so that for every sequence of reductions
\begin{equation}
    \pi_1 \stackrel{\gamma_1}{\lto} \ldots \stackrel{\gamma_{n-1}}{\lto} \pi_{n-1}
\end{equation}
the morphisms $T_{\gamma_i}$ factor through the appropriate restrictions so that we end up with a composable sequence of morphisms
\begin{equation}
    T_{\gamma_1} \circ \ldots \circ T_{\gamma_{n-1}}.
\end{equation}
\end{remark}

\begin{thm}
\label{thm:main}
If $\gamma: \pi \lto \pi'$ is a reduction, then the morphisms $S_{\gamma}: \bb{S}(\pi) \lto \bb{S}(\pi'), T_{\gamma}: \bb{Y}(\pi') \lto \bb{S}(\pi)$ restrict to well defined morphisms
\begin{align*}
S_{\gamma}\vert_{\bb{X}(\pi)}: \bb{X}(\pi) &\lto \bb{X}(\pi')\\
T_{\gamma}\vert_{\bb{X}(\pi')}: \bb{X}(\pi') &\lto \bb{X}(\pi)
\end{align*}
which are mutually inverse isomorphisms.
\end{thm}
\begin{proof}
\textbf{$\gamma: \pi \lto \pi'$ is an $\ax/\cut$-reduction.}
We refer to Diagram \eqref{eq:ax/cut_red} and consider only this type of $\ax/\cut$-reduction.

It suffices to consider only the links involved in the reduction. Let $l$ denote the link in $\pi$ to which $\neg A_c$ is the conclusion, and let $l'$ denote the link in $\pi$ to which $\neg A_a$ is the premise. We define the following restrictions
\begin{align*}
S_{\gamma, l} &= S_{\gamma}\vert_{\Delta_{a,b} \cap \Delta_{b,c} \cap \bb{X}(l)}\\
S_{\gamma, l'} &= S_\gamma\vert_{\bb{X}(l') \cap \Delta_{a,b} \cap \Delta_{b,c}}
\end{align*}
and consider the following three dimensional diagram, ignoring the dashed line for now:
\begin{equation}
\begin{tikzcd}[column sep = small]
\label{eq:the_axiom_cube}
    & {\bb{X}(l') \cap \bb{}X(l)} && {\bb{X}(l)} \\
    {\bb{X}(l') \cap \Delta_{a,b} \cap \Delta_{b,c} \cap \bb{X}(l)} && {\Delta_{a,b} \cap \Delta_{b,c} \cap \bb{X}(l)} \\
    & {\bb{X}(l')} && {\bb{S}(\pi')} \\
    {\bb{X}(l') \cap \Delta_{a,b} \cap \Delta_{b,c}} && {\bb{S}(\pi)}
    \arrow[from=1-2, to=1-4]
    \arrow[from=1-2, to=3-2]
    \arrow[from=1-4, to=3-4]
    \arrow[dashed, from=2-1, to=1-2]
    \arrow[from=2-1, to=2-3]
    \arrow[from=2-1, to=4-1]
    \arrow["{S_{\gamma, l}}"', from=2-3, to=1-4]
    \arrow[from=2-3, to=4-3]
    \arrow[from=3-2, to=3-4]
    \arrow["{S_{\gamma, l'}}", from=4-1, to=3-2]
    \arrow[from=4-1, to=4-3]
    \arrow["{S_\gamma}"', from=4-3, to=3-4]
\end{tikzcd}
\end{equation}
The morphisms $S_{\gamma, l}$ and $S_{\gamma, l'}$ are isomorphisms with inverses given by $T_\gamma\vert_{\bb{X}(l)}, T_{\gamma}\vert_{\bb{X}(l')}$ respectively. The front-face and the back-face of the cube \eqref{eq:the_axiom_cube}, given as follows, are both pullback diagrams:
\[\begin{tikzcd}
    {\bb{X}(l') \cap \Delta_{a,b} \cap \Delta_{b,c} \cap \bb{X}(l)} & {\Delta_{a,b} \cap \Delta_{b,c} \cap \bb{X}(l)} & {\bb{X}(l') \cap \bb{}X(l)} & {\bb{X}(l)} \\
    {\bb{X}(l') \cap \Delta_{a,b} \cap \Delta_{b,c}} & {\bb{S}(\pi)} & {\bb{X}(l')} & {\bb{S}(\pi')}
    \arrow[from=1-1, to=1-2]
    \arrow[from=1-1, to=2-1]
    \arrow[from=1-2, to=2-2]
    \arrow[from=1-3, to=1-4]
    \arrow[from=1-3, to=2-3]
    \arrow[from=1-4, to=2-4]
    \arrow[from=2-1, to=2-2]
    \arrow[from=2-3, to=2-4]
\end{tikzcd}\]
This implies that the dashed arrow $\bb{X}(l') \cap \Delta_{a,b} \cap \Delta_{b,c} \cap \bb{X}(l) \lto \bb{X}(l') \cap \bb{X}(l)$ in \eqref{eq:the_axiom_cube} exists, and is an isomorphism with inverse given by $T_{\gamma}\vert_{\bb{X}(l') \cap \bb{X}(l)}$.

\textbf{$\gamma: \pi \lto \pi'$ is a $\otimes/\parr$-reduction.}
This case is similar to the previous so we omit the proof.

\textbf{$\gamma: \pi \lto \pi'$ is a $!/?$-reduction.}
We consider only the case where there is a restricted amount of Pax-links, and with Conclusion-links as displayed in Definition \ref{def:S_T}, but the general result follows easily from this.

We refer to Diagram \eqref{eq:?=!-red}. Let $\zeta$ denote the proof net inside the box. We have already seen in Definition \ref{def:MELL_interpretation} that if we write $\bb{S}(?B) = \coprod_{h \in \mathcal{H}} \bb{P}^{s_h}$, fix a Hilbert function $h_1 \in \mathcal{H}$, denote the number of unoriented atoms of $B$ by $m_1$, denote the number of unoriented atoms of $A$ by $m$, and let
\begin{equation}
S_1 = \mathbbm{k}[x_0, \ldots, x_{2^{m_1}-1}],\quad S = \mathbbm{k}[x_1, \ldots, x_{2^m - 1}],
\end{equation}
then we can construct a morphism
\begin{equation}
f: H_{S_1}^{h_1} \lto H_S^h
\end{equation}
for some Hilbert function $h$. We have also shown in Definition \ref{def:MELL_interpretation} how to construct a morphism $\bb{U}_h \lto H_S^h \times \bb{S}(A)$ which we post-compose with the product of the composite $H_S^h \lto \bb{P}^{s_h} \lto \bb{S}(!A)$ and the identity on $\bb{S}(A)$ to obtain $\iota_h: \bb{U}_h \lto \bb{S}(!A) \times \bb{S}(A)$. On the other hand, consider the closed immersion
\begin{equation}
\label{eq:dirty_writing_trick}
\Gamma_f \lto H_{S_1}^{h_1} \times H_S^h
\end{equation}
of the graph $\Gamma_f$ of $f$. Associated to $h_1, h$ are fixed choices of closed immersions $H_{S_1}^{h_1} \lto \bb{P}^{s_{h_1}}, H_S^h \lto \bb{P}^{s_h}$ which we can post-compose with the canonical inclusions to obtain $H_{S_1}^{h_1} \lto \bb{S}(?B), H_S^h \lto \bb{S}(!A)$. Post-composing \eqref{eq:dirty_writing_trick} with the product of these yields
\begin{equation}
o: \Gamma_f \lto \bb{S}(?B) \times \bb{S}(!A).
\end{equation}
Denote by $\rho: \bb{S}(\zeta) \lto \bb{S}(?B) \times \bb{S}(A)$ the canonical projection and $\rho_\ast$ the pushforward. We claim that the following is a pullback diagram:
\begin{equation}
\label{eq:actual_pullback}
\begin{tikzcd}
	{\rho_\ast(\bb{X}(\zeta)) \cap \Gamma_f} && {\bb{U}_h \times \bb{S}(?B)} \\
	{\bb{S}(A) \times \Gamma_f} && {\bb{S}(A) \times \bb{S}(?B) \times \bb{S}(!A)}
	\arrow[from=1-1, to=1-3]
	\arrow[from=1-1, to=2-1]
	\arrow["{\iota_h \times \operatorname{id}}"', from=1-3, to=2-3]
	\arrow["{\operatorname{id} \times o}", from=2-1, to=2-3]
\end{tikzcd}
\end{equation}
It suffices to show that the following is a pullback diagram:
\[\begin{tikzcd}
	{\rho_\ast(\bb{X}(\zeta))} & {\bb{U}_h} \\
	{H_{S_1}^{h_1} \times \bb{S}(A)} & {H_S^h \times \bb{S}(A)}
	\arrow[from=1-1, to=1-2]
	\arrow[from=1-1, to=2-1]
	\arrow["{\iota_h}", from=1-2, to=2-2]
	\arrow["{f \times \operatorname{id}}"', from=2-1, to=2-2]
\end{tikzcd}\]
This can be shown by taking open affine charts of $H_{S_1}^{h_1}, H_S^h$ and using the fact that the tensor product induces pullbacks in the category $\mathbbm{k}-\underline{\operatorname{Alg}}$ of $\mathbbm{k}$-algebras.

\textbf{$\gamma: \pi \lto \pi'$ is a $\weak/!$-reduction.} This case is trivial as we are mapping empty schemes to empty schemes via morphisms uniquely defined by the property that their domain is the initial object in the category of schemes over $\mathbbm{k}$.

\textbf{$\gamma: \pi \lto \pi'$ is a $\ctr/!$-reduction.} We refer to Diagram \eqref{eq:ctr/!_reduction}. Due to the diagonals at the Axiom- and Cut-links it suffices to consider only the displayed Promotion-links, Pax-links, and the displayed Contraction-link of $\pi'$.

Let $l_!, l_{\pax}$ respectively denote the displayed Promotion and Pax-links of $\pi$. Let $l_!^L$, $l_{\pax}^L$, $l_!^R$, $L_{\pax}^R$, $l_{\ctr}$ respectively denote the Promotion-link of $\pi'$ displayed on the left, the Pax-link of $\pi'$ displayed on the left, the Promotion-link of $\pi'$ displayed on the right, the Pax-link of $\pi'$ displayed on the right, and the Contraction-link of $\pi'$.

By inspection of the definition of $S_\gamma, T_\gamma$, we obtain the following commuting diagram
\[\begin{tikzcd}[column sep = small]
    {\bb{X}(l_!) \cap \bb{X}(l_{\pax})} & {\bb{X}(l_!^L) \cap \bb{X}(l_{\pax}^L) \cap \bb{X}(l_!^R) \cap \bb{X}(l_{\pax}^R) \cap \bb{X}(l_{\ctr})} & {\bb{X}(l_!) \cap \bb{X}(l_{\pax})} \\
    {\bb{X}(l_!) \cap \bb{X}(l_{\pax})} & {\Delta_{\bb{X}(l_!) \cap \bb{X}(l_{\pax})}} & {\bb{X}(l_!) \cap \bb{X}(l_{\pax})}
    \arrow[from=1-1, to=1-2]
    \arrow[from=1-2, to=1-3]
    \arrow[from=2-1, to=1-1]
    \arrow[from=2-1, to=2-2]
    \arrow["o", from=2-2, to=1-2]
    \arrow[from=2-2, to=2-3]
    \arrow[from=2-3, to=1-3]
\end{tikzcd}\]
where
\begin{equation}
\Delta_{\bb{X}(l_!) \cap \bb{X}(l_{\pax})} \lto \big(\bb{S}(A_c) \times \bb{S}(?B_d) \times \bb{S}(!A_g) \times \bb{S}(?B_h)\big)^2
\end{equation}
denotes the diagonal which factors through $o$. All vertical arrows are isomorphisms, and the bottom horizontal composition is the identity. The argument for the other composite is similar.
\end{proof}

\subsection{An example}
The Church numerals provide an interesting class of shallow proofs. It was explained in the Introduction how the exponential fragment of linear logic provides equations $x - \phi y, y - \psi z$, for example. In this setting the Church numeral $\underline{2}_X$ (for $X$ an atomic formula) will give rise to these exact equations once appropriate localisations have been chosen. Moreover we obtain the equation $\phi - \psi$. A consequence of setting these equations to zero is that $x = \phi^2 z$, where the power of 2 reflects the fact that we took the Church numeral two. Indeed, the Church numeral $\underline{n}_X$ gives rise to the equation $x = \phi^n z$.

\subsubsection{Cutting $\underline{2}_X$ against $\underline{0}_X$}
\label{ex:Grassmann_example}
Consider the following proof net $\pi$, which is the Church numeral $\underline{2}_X$ cut against a simple proof net given by appending a Promotion-link to the Church numeral $\underline{0}_X$. We have labelled the formulas artificially; each $X_p$ means the atomic formula $X$.
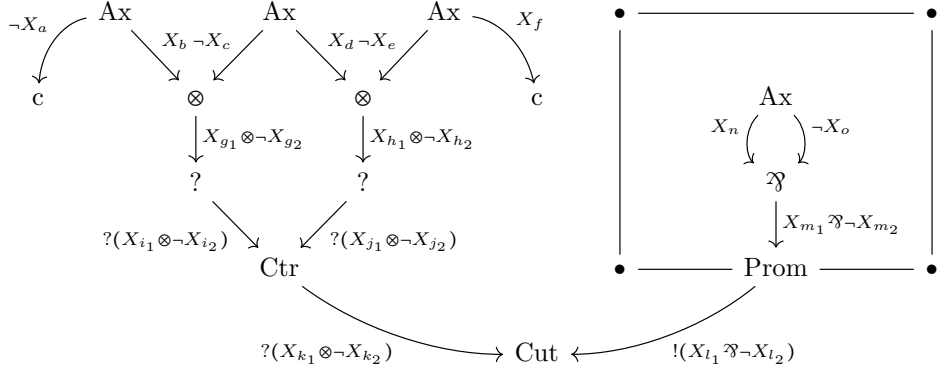
\[\begin{tikzcd}[column sep = small]
    & \ax && \ax && \ax && \bullet &&&&&& \bullet \\
    {\operatorname{c}} && \otimes && \otimes && {\operatorname{c}} &&&& \ax \\
    && {?} && {?} &&&&&& \parr \\
    &&& {\ctr} &&&& \bullet &&& {\prom} &&& \bullet \\
    &&&&&& \cut
    \arrow["{\neg X_a}"', curve={height=12pt}, from=1-2, to=2-1]
    \arrow["{X_b}", from=1-2, to=2-3]
    \arrow["{\neg X_c}"', from=1-4, to=2-3]
    \arrow["{X_d}", from=1-4, to=2-5]
    \arrow["{\neg X_e}"', from=1-6, to=2-5]
    \arrow["{X_f}", curve={height=-12pt}, from=1-6, to=2-7]
    \arrow[no head, from=1-8, to=1-14]
    \arrow[no head, from=1-8, to=4-8]
    \arrow["{X_{g_1} \otimes \neg X_{g_2}}", from=2-3, to=3-3]
    \arrow["{X_{h_1} \otimes \neg X_{h_2}}", from=2-5, to=3-5]
    \arrow["{\neg X_o}", curve={height=-12pt}, from=2-11, to=3-11]
    \arrow["{X_n}"', curve={height=12pt}, from=2-11, to=3-11]
    \arrow["{?(X_{i_1} \otimes \neg X_{i_2})}"', from=3-3, to=4-4]
    \arrow["{?(X_{j_1} \otimes \neg X_{j_2})}", from=3-5, to=4-4]
    \arrow["{X_{m_1} \parr \neg X_{m_2}}", from=3-11, to=4-11]
    \arrow["{?(X_{k_1} \otimes \neg X_{k_2})}"', curve={height=12pt}, from=4-4, to=5-7]
    \arrow[no head, from=4-8, to=4-11]
    \arrow[no head, from=4-11, to=4-14]
    \arrow["{!(X_{l_1} \parr \neg X_{l_2})}", curve={height=-12pt}, from=4-11, to=5-7]
    \arrow[no head, from=4-14, to=1-14]
\end{tikzcd}\]
Associated to the Axiom-links are the following projective schemes:
\begin{equation*}
\bb{S}(\neg X_a) = \bb{S}(X_b) = \bb{S}(\neg X_c) = \bb{S}(X_d) = \bb{S}(\neg X_e) = \bb{S}(X_f) = \bb{S}(X_n) = \bb{S}(\neg X_o) = \bb{P}^1.
\end{equation*}
For the Tensor- and Par-links, we should be considering the Segre embedding $\bb{P}^1 \times \bb{P}^1 \lto \bb{P}^3$, so that the interpretation of each non-atomic linear formula of $\pi$ is $\bb{P}^3$ but to make the ideas of the model more transparent within this example, we will directly consider the closed subscheme $\bb{P}^1 \times \bb{P}^1$:
\begin{equation*}
\bb{S}(X_{g_1} \otimes \neg X_{g_2}) = \bb{S}(X_{h_1} \otimes \neg X_{h_2}) = \bb{S}(X_{m_1} \parr \neg X_{m_2}) = \bb{P}^1 \times \bb{P}^1.
\end{equation*}
For each of the linear formulas we consider a corresponding graded $\mathbbm{k}$-algebra. For instance, associated to the formula $\neg X_a$ is the graded $\mathbbm{k}$-algebra $\mathbbm{k}[X_a', X_a]$. Since the variable $X$ is consistent throughout all of the formulas, we will use the algebra $\mathbbm{k}[a',a]$ in place of $\mathbbm{k}[X_a', X_a]$, and similarly for the other variables.

Let $S = \mathbbm{k}[z_1', z_1] \times \mathbbm{k}[z_2', z_2]$. Then for each Hilbert function $h \in \mathcal{H}$ we have a fixed choice of closed immersion of the Hilbert scheme $H_S^h$ given by Proposition \ref{prop:big_assumption} into some projective space $\bb{P}^{s_h}$. We take the disjoint union of the codomains of these:
\begin{align*}
\bb{S}(?(X_{i_1} \otimes \neg X_{i_2})) &= \bb{S}(?(X_{j_1} \otimes \neg X_{j_2}))\\
= \bb{S}(?(X_{k_1} \otimes \neg X_{k_2})) &= \bb{S}(!(X_{l_1} \otimes \neg X_{l_2})) = \coprod_{h \in \mathcal{H}}\bb{P}^{s_h}.
\end{align*}
The interior of the box of $\pi$ determines a point in $\coprod_{h \in \mathcal{H}}\bb{P}^{s_h}$ which is inside the closed subscheme $H_S^{h^\ast} \subseteq \bb{P}^{s_h}$ for some particular Hilbert function $h^\ast$. So, for this example we can restrict to the particular connected component of the Hilbert scheme determined by this Hilbert function, which we now calculate.

Consider $\underline{0}_X$
\[\begin{tikzcd}[column sep = small]
    \ax \\
    \parr \\
    {\operatorname{c}}
    \arrow["{X_n}"', curve={height=12pt}, from=1-1, to=2-1]
    \arrow["{\neg X_o}", curve={height=-12pt}, from=1-1, to=2-1]
    \arrow["{X_{m_1} \parr \neg X_{m_2}}", from=2-1, to=3-1]
\end{tikzcd}\]
We build the closed subscheme $\bb{X}(\underline{0}_X)$. From the Axiom-link $l_{\ax}$ we have the diagonal $\bb{X}(l_{\ax}) = \Delta_{n,o} \lto \bb{P}^1 \times \bb{P}^1$. If we consider the projection $\rho: \bb{S}(\underline{0}_X) \lto \bb{P}^1 \times \bb{P}^1$ then the composite
\begin{equation}
\bb{X}(\underline{0}_X) \lto \bb{S}(\underline{0}_X) \lto \bb{S}(X_{m_1} \parr \neg X_{m_2}) = \bb{P}^1 \times \bb{P}^1
\end{equation}
is isomorphic to the diagonal
\begin{equation}
\bb{P}^1 \stackrel{\Delta}{\lto} \bb{P}^1 \times \bb{P}^1.
\end{equation}
The ideal
\begin{equation}
\label{eq:complicated_diagonal}
I = (m_1m_2' - m_1'm_2) \subseteq \mathbbm{k}[m_1',m_1] \times_{\mathbbm{k}} \mathbbm{k}[m_2', m_2],
\end{equation}
with $S = \mathbbm{k}[m_1',m_1] \times_{\mathbbm{k}}\mathbbm{k}[m_2',m_2]$,
is such that $\operatorname{Proj}(S/I) \cong \bb{X}(\underline{0}_X)$ (as closed subschemes of $\bb{P}^1 \times \bb{P}^1$). Conceptually, this ideal may be thought of as its corresponding counterpart obtained by dividing by the primed variables. More specifically, we have
\begin{equation}
\Big(\frac{m_1m_2'}{m_1'm_2'} - \frac{m_1'm_2}{m_1'm_2'}\Big) \subseteq S_{(m_1'm_2')}.
\end{equation}
Carrying this through the composition of $\mathbbm{k}$-algebra isomorphisms 
\begin{equation}
    S_{(m_1'm_2')} \lto \mathbbm{k}[m_1/m_1',m_2/m_2'] \lto \mathbbm{k}[m_1, m_2]
\end{equation}
determined by the rules
\begin{equation}
\frac{m_1m_2'}{m_1'm_2'} \longmapsto m_1/m_1' \lto m_1,\quad \frac{m_1'm_2}{m_1'm_2'} \longmapsto m_2/m_2' \lto m_2
\end{equation}
we obtain the ideal
\begin{equation}
\label{eq:revealing_diagonal}
(m_1 - m_2) \subseteq \mathbbm{k}[m_1, m_2].
\end{equation}
So, we can think of \eqref{eq:complicated_diagonal} as the equation ``$m_1 = m_2$'', see Section \ref{sec:MLL_relation} for more details.

We saw in Example \ref{ex:Gotmann_number_ex} that the Hilbert function of $I \subseteq S$ is $h^\ast(d) = 2d + 1$, and that the Gotzmann number $G(I)$ of $I$ is 2. The degree $2$ component of the algebra corresponding to $\bb{P}^3$ maps onto the degree $1$ component of the algebra corresponding to $\bb{P}^1 \times \bb{P}^1$. The Hilbert scheme $H_S^{h^\ast}$ is by Proposition \ref{prop:big_assumption} therefore a closed subscheme of $G_{S_{1}}^{h^\ast(1)} = G_{4}^3$. We identify the degree $1$ component $S_1$ of $S$ with $\mathbbm{k}^{4}$ via the isomorphism defined by linearity and the following assignments, where $e_1, \ldots, e_{4}$ are the standard basis vectors for $\mathbbm{k}^{4}$:
\begin{align}
\label{eq:iso_choice}
m_1'm_2' &\longmapsto e_1 & m_1m_2' &\longmapsto e_2 & m_1'm_2 &\longmapsto e_3 & m_1m_2 &\longmapsto e_4.
\end{align}
Let $\eta: H_S^{h^\ast} \lto G_{4}^3$ denote this closed immersion. Recall from Lemma \ref{lem:Grassmann_local_rep} that for any size $k$ subset $B$ of $\{e_1, \ldots, e_{4}\}$ the open subset $G_{4\backslash B}^3 \subseteq G_{4}^3$ is representable. Consider the set $B = \{e_3\}$ which corresponds to $\{m_1'm_2\} \subseteq S_1$. There exists the following pullback diagram:
\[\begin{tikzcd}
    {\eta^{-1}(H^{h^\ast}_S)} & {G_{4\backslash B}^3} \\
    {H_{S}^{h^\ast}} & {G_{4}^3}
    \arrow["{\hat{\eta}}", from=1-1, to=1-2]
    \arrow[from=1-1, to=2-1]
    \arrow[from=1-2, to=2-2]
    \arrow["\eta", from=2-1, to=2-2]
\end{tikzcd}\]
By Lemma \ref{lem:Grassmann_local_rep} $G_{4\backslash B}^3$ is represented by $\operatorname{Spec}\mathcal{A}$ where $\mathcal{A}$ is the following ring:
\begin{equation}
\mathcal{A} = \mathbbm{k}[\{y_{i} \mid 1 \leq i \leq 3\}].
\end{equation}
Since $\eta$ is a closed immersion, it follows that $\hat{\eta}$ is and so there exists an ideal $\mathcal{J} \subseteq S$ such that $\eta^{-1}(H_S^{h^\ast}) \cong \operatorname{Spec}\mathcal{A}/\mathcal{J}$. By representability of $\underline{H_S^{h^\ast}}$ and $\underline{G_{4}^3}$ the morphism $\eta$ corresponds to a natural transformation $\underline{\eta}$ between functors $\underline{\eta}: \underline{H_S^{h^\ast}} \lto \underline{G_{4}^3}$. Let $R$ be a $\mathbbm{k}$-algebra, the function $\underline{\eta}_R: \underline{H_S^{h^\ast}}(R) \lto \underline{G_{4}^3}(R)$ maps a homogeneous ideal $L \subseteq R \otimes_{\mathbbm{k}}S$ to the submodule $L_1 \subseteq R \otimes_{\mathbbm{k}}S_1 \cong R^{4}$. A homomorphism $\mathcal{A}/\mathcal{J} \lto R$ is given by a collection of coefficients $\{\alpha_{p} \in R\}_{1 \leq p \leq 3}$ satisfying the equations of $\mathcal{J}$. These equations determine the Hilbert scheme as a subscheme of the Grassmann scheme, and so for the sake of simplicity we can ignore them and deal only with the coefficients $\{\alpha_{p} \in R\}_{1 \leq p \leq 3}$, i.e. $\mathbbm{k}$-algebra homomorphisms $\mathcal{A} \lto R$, i.e. points of the Grassmann scheme $\operatorname{Spec}R \lto G_{4\backslash B}^3$.

We have the following equation in $\mathcal{A}/I$ by \eqref{eq:iso_choice}.
\begin{align*}
m_1m_2' = m_1'm_2
\end{align*}
which corresponds to the following subspace of $\mathbbm{k}^{4}$:
\begin{equation}
\operatorname{Span}_{\mathbbm{k}}\{e_2 - e_3 \}.
\end{equation}
So, we define a function $\mathcal{A} \lto \mathbbm{k}$ as the $\mathbbm{k}$-algebra homomorphism generated by the following rules:
\begin{equation}
y_{1} \longmapsto 0,\quad y_2 \longmapsto 1, \quad y_3 \longmapsto 0.
\end{equation}
These equations come from the fact that in $\mathbbm{k}^4/\operatorname{Span}_{\mathbbm{k}}\{e_2 - e_3 \}$ we have the equation $e_3 = y_0 e_1 + y_1 e_2 + y_3 e_4$, if $y_1 = y_3 = 0, y_2 = 1$. Thus, the ideal corresponding to the Promotion-link is
\begin{equation}
(y_1, y_2 - 1, y_3) \subseteq \mathcal{A}.
\end{equation}
Now we consider the Dereliction-link:
\[\begin{tikzcd}
    \vdots \\
    {?} \\
    \vdots
    \arrow["{X_{g_1} \otimes \neg X_{g_2}}", from=1-1, to=2-1]
    \arrow["{?(X_{i_1} \otimes \neg X_{i_2})}", from=2-1, to=3-1]
\end{tikzcd}\]
For the Hilbert function $h^\ast$ we have again that $G_{4\backslash B}^3$ is represented by another copy of $\mathcal{A}$:
\begin{equation}
\mathcal{A}' = \mathbbm{k}[y_{i}' \mid 1 \leq i \leq 3].
\end{equation}
There is a universal subspace of $(\mathcal{A}')^{4}$ with basis $B$ given as follows:
\begin{equation}
\label{eq:Universal_dereliction}
\operatorname{Span}_{\mathcal{A}'}\{e_3 - y_1'e_1 - y_2'e_2 - y_3'e_4\} \subseteq (\mathcal{A}')^{4}.
\end{equation}
This translates through \eqref{eq:iso_choice} (with $m_1',m_1,m_2',m_2$ respectively replaced by $g_1',g_1,g_2',g_2$) to
\begin{equation}
(g_1'g_2 - y_1' g_1'g_2' - y_2' g_1g_2' - y_3' g_1g_2) \subseteq \mathcal{A}'[g_1', g_1] \times_{\mathbbm{k}} \mathcal{A}'[g_2', g_2].
\end{equation}
Similarly, for the other Dereliction-link we have a third copy of $\mathcal{A}$:
\begin{equation}
\mathcal{A}'' = \mathbbm{k}[y_i'' \mid 1 \leq i \leq 3]
\end{equation}
and the universal subspace
\begin{equation}
\operatorname{Span}_{\mathcal{A}''}\{e_3 - y_1''e_1 - y_2''e_2 - y_3''e_4)\} \subseteq (\mathcal{A}'')^{4}
\end{equation}
with corresponding ideal
\begin{equation}
(h_1'h_2 - y_1'' h_1'h_2' - y_2'' h_1h_2' - y_3'' h_1h_2) \subseteq \mathcal{A}''[h_1', h_1] \times_{\mathbbm{k}} \mathcal{A}''[h_2', h_2].
\end{equation}
The Contraction-link introduces a fourth copy of $\mathcal{A}$:
\begin{equation}
\mathcal{A}''' = \mathbbm{k}[y_i''' \mid 1 \leq i \leq 3]
\end{equation}
and contributes the following ideal
\begin{equation}
(y'_i - y'''_i, y''_i - y'''_i)_{1 \leq i \leq 3} \subseteq \mathcal{A}' \otimes_{\mathbbm{k}} \mathcal{A}'' \otimes_{\mathbbm{k}}\mathcal{A}'''.
\end{equation}
Finally the Cut-link contributes the ideal
\begin{equation}
(y'''_i - y_i)_{1 \leq i \leq 3} \subseteq \mathcal{A}''' \otimes_{\mathbbm{k}} \mathcal{A}.
\end{equation}
All that remains to be considered is the linear component of the proof. The Axiom-link with conclusions $\neg X_a, X_b$ is interpreted as the diagonal $\Delta \lto \bb{P}^1 \times \bb{P}^1$ which is given by the following ideal:
\begin{equation}
(ab' - a'b) \subseteq \mathbbm{k}[a', a] \times_{\mathbbm{k}} \mathbbm{k}[b',b].
\end{equation}
The other Axiom-links are treated similarly.

The Tensor-link with conclusion $X_{g_1} \otimes \neg X_{g_2}$ contributes the following ideal:
\begin{align*}
&(g_1g_2'b'c' - g_1'g_2'bc', g_1'g_2b'c' - g_1'g_2'b'c, g_1g_2b'c' - g_1'g_2'bc)\\
&\subseteq \big(\mathbbm{k}[g_1',g_1] \times_{\mathbbm{k}} \mathbbm{k}[g_2', g_2]\big) \times_{\mathbbm{k}} \mathbbm{k}[b', b] \times_{\mathbbm{k}} \mathbbm{k}[c',c].
\end{align*}
Again, we can think of this as the corresponding ideal given by dividing by the primed variables, given as follows:
\begin{equation}
(g_1 - b, g_2 - c, g_1g_2 - bc) = (g_1 - b, g_2 - c) \subseteq \mathbbm{k}[g_1, g_2, b, c].
\end{equation}
This reflects the logical structure that the premises $X_b,\neg X_c$ of the Tensor-link have respective corresponding conclusions $X_{g_1}, \neg X_{g_2}$.

The other Tensor-link and the Par-link are similar. We thus have the following set of equations:
\begin{align*}
ab' - a'b,\quad cd' - c'd,\quad ef'-e'f,\quad no' - n'o,
\end{align*}
\begin{align*}
g_1g_2'b'c' &- g_1'g_2'bc', & g_1'g_2b'c' &- g_1'g_2'b'c, & g_1g_2b'c' &- g_1'g_2'bc,\\
h_1h_2'd'e' &- h_1'h_2'de', & h_1'h_2d'e' &- h_1'h_2' d'e, & h_1h_2d'e' &- h_1'h_2'de,\\
m_1m_2'n'o' &- m_1'm_2'no', & m_1'm_2n'o' &- m_1'm_2'n'o, & m_1m_2n'o' &- m_1'm_2'no,
\end{align*}
\begin{equation*}
g_1'g_2 - y_1' g_1'g_2' - y_2' g_1g_2' - y_3' g_1g_2,
\end{equation*}
\begin{equation*}
h_1'h_2 - y_1'' h_1'h_2' - y_2'' h_1h_2' - y_3'' h_1h_2,
\end{equation*}
\begin{equation*}
y_{1},\quad y_{2} - 1,\quad y_{3},
\end{equation*}
\begin{align*}
& y_i - y'''_i,\text{ for }1 \leq i \leq 3,\\
& y'_i - y'''_i,\text{ for }1 \leq i \leq 3,\\
& y''_i - y'''_i,\text{ for }1 \leq i \leq 3.
\end{align*}
To understand these equations, we can localise at all of the primed variables (except for the $y$ variables) and obtain the following set of polynomials:
\begin{equation*}
a - b,\quad c - d,\quad e - f,\quad n - o,
\end{equation*}
\begin{align*}
g_1 &- b, \quad g_2 - c,\\
h_1 &- d, \quad h_2 - e,\\
m_1 &- n, \quad m_2 - o,
\end{align*}
\begin{equation*}
g_2 - y_1' - y_2'g_1 - y_3'g_1g_2,
\end{equation*}
\begin{equation*}
h_2 - y_1'' - y_2''h_1 - y_3''h_1h_2,
\end{equation*}
\begin{equation*}
y_{1},\quad y_{2} - 1,\quad y_{3},
\end{equation*}
\begin{align*}
& y_i - y'''_i,\text{ for }1 \leq i \leq 3,\\
& y'_i - y'''_i,\text{ for }1 \leq i \leq 3,\\
& y''_i - y'''_i,\text{ for }1 \leq i \leq 3.
\end{align*}

\begin{remark}
\label{rmk:eq_btw_eq}
The deduction rules involving the exponentials contribute ``equations between equations'' to the geometry of $\bb{X}(\pi)$ in a sense we now explain. Consider the two equations
\begin{align}
\label{eq:equations_(between_equations)}
g_2 &- y_1' - y_2'g_1 - y_3'g_1g_2,\\
h_2 &- y_1'' - y_2''h_1 - y_3''h_1h_2.
\end{align}
These are the equations pertaining to the Dereliction-links of $\pi$. The variables
\begin{equation}
y_1',y_2',y_3',y_1'',y_2'',y_3''
\end{equation}
have constraints put upon them by the Contraction-link which introduces the following:
\begin{align}
y_1' &= y_1''', & y_2' &= y_2''', & y_3' &= y_3'''\\
y_1'' &= y_1''', & y_2'' &= y_2''', & y_3'' &= y_3'''
\end{align}
which imposes the following equations in the quotient:
\begin{align}
y_1' &= y_1'', & y_2' &= y_2'', & y_3' &= y_3''.
\end{align}
That is, the normal vector of the two linear spaces \eqref{eq:equations_(between_equations)} are set to be equal via the Contraction-link. Notice that this does \emph{not} impose that $g_2 = h_1$. This equation \emph{does} hold, but due to the Axiom-link with conclusions $\neg X_c, X_d$, and the two Tensor-links outside of the box. So, the exponential fragment of shallow proofs only make identifications between the \emph{coefficients} of polynomials. The linear component of the proof makes identifications between the variables.
\end{remark}

\begin{remark}
\label{rmk:elimination}
It is interesting to note that Section \ref{ex:Grassmann_example} seems to extend the theory of \cite{AlgPnt} which relates cut-elimination to elimination theory. Define all variables pertaining to edges which lie above the Cut-link to be \emph{elimination variables}, and the remaining two variables $a, f$ to be \emph{non-elimination variables}. Using software algebra, performing the Buchberger Algorithm on the final set of polynomials given in the example yields the polynomial $a - f$, which is the result of localising the diagonal $\Delta \lto \bb{P}^1 \times \bb{P}^1$ which is the closed immersion corresponding to the normal form of $\pi$:
\[\begin{tikzcd}[column sep = small]
    & \ax \\
    {\operatorname{c}} && {\operatorname{c}}
    \arrow["{\neg X_a}"', curve={height=12pt}, from=1-2, to=2-1]
    \arrow["{X_f}", curve={height=-12pt}, from=1-2, to=2-3]
\end{tikzcd}\]
\end{remark}

\subsection{Relation to MLL model}
\label{sec:MLL_relation}
In \cite{AlgPnt} we attributed to an Axiom-link
\begin{equation}
\label{eq:ax_link}
\begin{tikzcd}[column sep = small]
	& \ax \\
	\vdots && \vdots
	\arrow["{\neg X_1}"', curve={height=12pt}, from=1-2, to=2-1]
	\arrow["{X_2}", curve={height=-12pt}, from=1-2, to=2-3]
\end{tikzcd}
\end{equation}
the equation
\begin{equation}
    \label{eq:algpnt_equation}
    X_1 - X_2 \in \mathbbm{k}[X_1,X_2].
\end{equation}
In this paper, we have attributed to \eqref{eq:ax_link} the composite of closed immersions
\begin{equation}
    \Delta \lto \bb{P}^1 \times \bb{P}^1 \lto \bb{P}^3
\end{equation}
where $\Delta \lto \bb{P}^1 \times \bb{P}^1$ is the diagonal and $\bb{P}^1 \times \bb{P}^1 \lto \bb{P}^3$ is the Segre embedding. We remark that
\begin{equation}
\operatorname{Proj}\big((\mathbbm{k}[X_1',X_1] \times_{\mathbbm{k}} \mathbbm{k}[X_2', X_2])/(X_1X_2' - X_1'X_2)\big) \cong \Delta
\end{equation}
where we have made use of the cartesian product $\times_{\mathbbm{k}}$, see Definition \ref{def:cartesian_product}. Thus, algebraically, we can think of the present model as attributing to \eqref{eq:ax_link} the equation
\begin{equation}
\label{eq:this_paper_eq}
    X_1X_2' - X_1'X_2 \in \mathbbm{k}[X_1',X_1] \times_{\mathbbm{k}} \mathbbm{k}[X_2', X_2].
\end{equation}
We saw in Section \ref{ex:Grassmann_example} that we can recover \eqref{eq:algpnt_equation} from \eqref{eq:this_paper_eq} by dividing by the primed variables. That is, the composition of isomorphisms
\begin{equation}
    \Big(\mathbbm{k}[X_1',X_1] \times_{\mathbbm{k}}\mathbbm{k}[X_2', X_2]\Big)_{(X_1'X_2')} \lto \mathbbm{k}[X_1/X_1', X_2/X_2'] \lto \mathbbm{k}[X_1, X_2]
\end{equation}
determined by the rules
\begin{equation}
    \frac{X_1X_2'}{X_1'X_2'} \longmapsto X_1/X_1' \longmapsto X_1,\quad \frac{X_1'X_2}{X_1'X_2'} \longmapsto X_2/X_2' \longmapsto X_2
\end{equation}
maps \eqref{eq:this_paper_eq} to \eqref{eq:algpnt_equation}. Lemma \ref{lem:relation} generalises this observation to all of MLL and make precise the claim that Theorem \ref{thm:main} generalises \cite[Proposition 4.6]{AlgPnt}, we make use of the notation there.\\

Let $A$ be a linear formula with $n$ unoriented atoms $\{X_1, \ldots, X_n\}$ so that $P_A = \mathbbm{k}[X_1, \ldots, X_n]$, where $P_A$ is the polynomial ring associated to $A$ of \cite[Definition 3.14]{AlgPnt}. This is the coordinate ring of the scheme $\bb{A}_{\mathbbm{k}}^n$ which is an open subset of the product $(\bb{P}_{\mathbbm{k}}^1)^n$. Let $\pi$ be a proof in multiplicative linear logic.

\begin{defn}
    For all $n>0$ let $S_n$ denote the $\mathbbm{k}$-algebra
    \begin{equation}
        S_n = \mathbbm{k}[x_1,x_1'] \times_{\mathbbm{k}} \ldots \times_{\mathbbm{k}}\mathbbm{k}[x_n,x_n'].
    \end{equation}
    For each $n>0$, fix an isomorphism
    \begin{equation}
        \Phi_n: \operatorname{Proj}S_n \lto (\bb{P}_{\mathbbm{k}}^1)^n.
    \end{equation}
    We also fix isomorphisms
    \begin{equation}
        \Psi_n: \operatorname{Spec}(\mathbbm{k}[x_1,\ldots, x_n]) \lto \operatorname{Spec}\Big(\big(S_n)_{(x_1'\ldots x_n')}\Big).
    \end{equation}
    We have a composition
\begin{equation}
    \bb{A}_{\mathbbm{k}}^n = \operatorname{Spec}(\mathbbm{k}[x_1,\ldots, x_n]) \stackrel{\psi_n}{\lto} \operatorname{Spec}\Big(\big(S_n)_{(x_1'\ldots x_n')}\Big) \stackrel{\Phi_n}{\lto} (\bb{P}_{\mathbbm{k}}^1)^n \lto \bb{S}(A)
\end{equation}
where the final map is the Segre embedding. If $\pi$ is a proof in multiplicative linear logic, we take the product of these over all formulas $A_e$ labelling edges $e$ of $\pi$ to obtain an inclusion
\begin{equation}
    \iota_{\pi}:\operatorname{Spec}P_\pi \cong \prod_{e \in \mathcal{E}_\pi}\bb{A}_{\mathbbm{k}}^{|U_{A_e}|} \lto \bb{S}(\pi)
\end{equation}
where $P_\pi$ is the polynomial ring associated to $\pi$ as defined in \cite[Definition 3.14]{AlgPnt} $U_{A_e}$ denotes the set of unoriented atoms of $A_e$ as defined in \cite[Definition 3.7]{AlgPnt}.
\end{defn}

\begin{defn}
    Given a reduction $\gamma: \pi \lto \pi'$ of proofs $\pi, \pi'$ in multiplicative linear logic, let $S^\bullet, T^\bullet$ denote the $\mathbbm{k}$-algebra homomorphisms defined as $S,T$ in \cite[Proposition 4.6]{AlgPnt}.
\end{defn}

\begin{lemma}
\label{lem:relation}
    For any reduction $\gamma: \pi \lto \pi'$ where $\pi, \pi'$ are proofs in multiplicative linear logic, there exists the following commuting diagrams:
\[\begin{tikzcd}
	{\bb{S}(\pi)} && {\bb{S}(\pi')} & {\bb{S}(\pi')} && {\bb{S}(\pi)} \\
	{\operatorname{Spec}P_\pi} && {\operatorname{Spec}P_{\pi'}} & {\operatorname{Spec}P_{\pi'}} && {\operatorname{Spec}P_\pi}
	\arrow["S", from=1-1, to=1-3]
	\arrow["T", from=1-4, to=1-6]
	\arrow["{\iota_\pi}", from=2-1, to=1-1]
	\arrow["{\operatorname{Spec}T^\bullet}"', from=2-1, to=2-3]
	\arrow["{\iota_{\pi'}}"', from=2-3, to=1-3]
	\arrow["{\iota_\pi}", from=2-4, to=1-4]
	\arrow["{\operatorname{Spec}S^\bullet}"', from=2-4, to=2-6]
	\arrow["{\iota_{\pi'}}"', from=2-6, to=1-6]
\end{tikzcd}\]
\end{lemma}
\begin{proof}
    In this case, the maps $S, T$ are either projections or diagonal morphisms. Thus, for the statement to hold, one must check that $\operatorname{Spec}T^\bullet, \operatorname{Spec}S^\bullet$ are the appropriate restrictions of these. This is done by inspection along with (a generalisation of) the fact that the diagonal morphism $\Delta: \bb{A}_{\mathbbm{k}}^1 \lto \bb{A}_{\mathbbm{k}}^2$ is $\operatorname{Spec}(f)$ where $f: \mathbbm{k}[x_1, x_2] \lto \mathbbm{k}[x]$ is the $\mathbbm{k}$-algebra homomorphism mapping $x_1 \longmapsto x, x_2 \longmapsto x$, along with (a generalisation of) the fact that the projection (for $i = 1,2$) $\operatorname{Projection}_i: \bb{A}_{\mathbbm{k}}^2 \lto \bb{A}_{\mathbbm{k}}^1$ is $\operatorname{Spec}(g_i)$ where $g_i: \mathbbm{k}[x] \lto \mathbbm{k}[x_1, x_2]$ is the $\mathbbm{k}$-algebra homomorphism mapping $x \longmapsto x_i$.
\end{proof}

\section{Future paths}
\label{sec:future}
\textbf{Extending the model to all of MELL.} The immediate obstruction to extending our model to all of MELL is Lemma \ref{lem:well_defined} which only considers shallow proofs. We postulate that this lemma \emph{does} indeed generalise. To check this, one must check that the local freeness condition is satisfied by the immersion \eqref{eq:big_embedding}. However, it is not precisely \emph{this} lemma which would be generalised. A difficulty in working with our model is the fact that we have defined the Hilbert scheme $H_T$ to parameterise graded $\mathbbm{k}$-modules $T$. In fact, a more general Hilbert scheme exists which is parameterised by projective schemes instead. Similarly to $H_T$ the more general Hilbert scheme represents a functor. The reference for this is \cite{HilbertScheme}.
\begin{defn}
Suppose $S$ is a locally Noetherian scheme and $X \lto S$ is a projective scheme over $S$. The \textbf{Hilbert functor} from the category $\underline{\operatorname{Sch}}_{S}$ whose objects are locally Noetherian schemes over $S$ to the category $\underline{\operatorname{Set}}$ of sets is given by:
\begin{align*}
\operatorname{Hilb}_{X/S}: \underline{\operatorname{Sch}}_{S} &\lto \underline{\operatorname{Set}}\\
T &\longmapsto \{\text{Closed subschemes }Z \lto X \times_S T \mid\\
&\quad\quad Z \text{ is flat over }T\}.
\end{align*}
\end{defn}
\begin{thm}
There exists a scheme $\operatorname{Hilb}_{X/S}$ representing this functor.
\end{thm}
This is a more interesting functor for us because it yields a simple way of thinking about Pax-links: they give rise to the product of locally projective schemes $X_1 \times \ldots \times X_n$ over which the closed subscheme corresponding to the interior of the box must be flat. In particular, this will avoid all necessity of localisation inside Definition \ref{def:MELL_interpretation}, which we believe would lead to a more natural model.

Thus, our future work proposal is as follows: first we extend the model for shallow proofs to one where the Hilbert scheme $H_T$ is replaced by the more general Hilbert scheme $\operatorname{Hilb}_{X/S}$ above. Then, we would aim to extend the resulting model to all of MELL. Of course, one could also dream of going even further than MELL; additives, differential linear logic, etc.

\textbf{Relating this model to the MLL models given in Chapter \cite[Chapter IV]{troiani2024phd}.} We presented three models of MLL in total in \cite{troiani2024phd}. One gives proofs as linear equations, which has been extended to shallow proofs inside this paper (as well as in \cite[Chapter III]{troiani2024phd}). The other two give proofs as matrix factorisations, and proofs as quantum error correction codes respectively. Section \cite[Section 4.2.5]{troiani2024phd} relates the algebraic model to the quantum error correction code model via the matrix factorisation one, and so it would be interesting to see how the Hilbert scheme plays a role in these other models. For instance, the Hilbert scheme plays the role of a moduli space, in that it parameterises flat families of closed subschemes. It would be interesting to consider the moduli space of matrix factorisations as a possible model in the sense of \cite{fiore2024monoidal} of the exponential in linear logic.

\textbf{Classifying the Hilbert functions which arise from proofs.} Our model considers the set $\mathcal{H}$ of \emph{all} Hilbert functions throughout. Surely not all Hilbert functions arise from proofs. It would be interesting to find the subset $\mathcal{H}' \subseteq \mathcal{H}$ so that $h \in \mathcal{H}'$ if and only if there exists a proof $\pi$ such that the closed subscheme $\bb{X}(\pi) \lto \bb{S}(\pi)$ has Hilbert function $h$.

It has been noted in \cite[Remark 2.12]{troiani2024phd} that there is more to the Geometry of Interaction program than just modelling cut-elimination with a non-trivial effective procedure. One could also ask for a correctness criterion such as the long trip condition given in the Sequentialisation Theorem (due to Girard \cite{LinearLogic}) to be present in our model as well. It is possible that the classification of the Hilbert functions which arise from proofs has relevance to this line of research.

\textbf{Elimination Theory.} The algebraic model given in \cite{AlgPnt} not only gives an interpretation of proofs in MLL but also relates the cut-elimination process to the Buchberger algorithm. As mentioned in Remark \ref{rmk:elimination} it seems possible that this relationship extends to MELL, at least to shallow proofs. There are many connections to the construction of the (multigraded) Hilbert scheme of \cite{hilb} and syzygies, monomial ideals, Gr\"{o}bner bases, etc. These connections should be fully developed.

\label{Bibliography}
\bibliographystyle{plain}  
\bibliography{Bibliography}

\appendix

\end{document}